\documentclass[10pt]{article}
\usepackage[OT1]{fontenc}  	
\usepackage[utf8]{inputenc}
\usepackage[spanish,french,english,USenglish]{babel} 	
\usepackage{OldStandard}
\usepackage{amssymb,amsmath,amsthm,units,esint,bbm}	 				
\usepackage{color,graphicx} 			
\usepackage[small]{caption}
\usepackage{cases}

\makeatletter
\renewcommand\section{\@startsection {section}{1}{\z@}
{-30pt \@plus -1ex \@minus -.2ex}
{2.3ex \@plus.2ex}
{\normalfont\normalsize\bfseries\boldmath}}

\renewcommand\subsection{\@startsection{subsection}{2}{\z@}
{-3.25ex\@plus -1ex \@minus -.2ex}
{1.5ex \@plus .2ex}
{\normalfont\normalsize\bfseries\boldmath}}

\renewcommand{\@seccntformat}[1]{\csname the#1\endcsname. }
\makeatother

\newtheorem{theorem}{Theorem}

\newtheorem{corollary}{Corollary}

\DeclareMathOperator{\sgn}{sgn} 

\renewcommand{\Re}{\mathop{\rm Re}\nolimits}
\newcommand{\res}{\mathop{\rm res}\limits}

\newcommand{\be}{\begin{equation}}
\newcommand{\ee}{\end{equation}} 
\makeatletter
\newcommand{\specialnumber}[1]{
\def\tagform@##1{\maketag@@@{(\ignorespaces##1\unskip\@@italiccorr#1)}}}
\makeatother

\makeatletter
\def\blfootnote{\xdef\@thefnmark{}\@footnotetext}
\def\bkfootnote{\xdef\@thefnmark{*}\@footnotetext}
\makeatother

\begin{document}

\begin{center}

{\bf \vspace{-6em}Three Notes on Ser's and Hasse's Representations \\[5pt] for the Zeta-Functions}
\vskip 20pt
{\bf Iaroslav V.~Blagouchine\footnote{Fellow of the Steklov Institute of Mathematics at St.~Petersburg,
Russia.\\[-0.3em]}}\\
{\small \it SeaTech, University of Toulon, France.}\\
{\small \it St.~Petersburg State University of Architecture and Civil Engineering, Russia.}\\
{\small \tt iaroslav.blagouchine@univ-tln.fr, iaroslav.blagouchine@pdmi.ras.ru}
\end{center}
\vskip 15pt

\blfootnote{\tt \underline{Note to the readers of the 5th arXiv version:} this version is a
copy of the journal version of the article, which has been published in INTEGERS,
Electronic Journal of Combinatorial Number Theory, 
vol.~18A (special volume in Honor of Jeffrey Shallit on the occasion of his 60th Birthday), article~\#A3, pp.~1-45, 2018.
Submission history: received 5 May 2017, accepted: 2 February 2018, published: 16 March 2018
(the first version of this article was published on arXiv on 7 June 2016). 
For any further reference to the material published here, 
please always use the journal version of the paper.}

\centerline{\bf Abstract}
\noindent
This paper is devoted to Ser's and Hasse's series representations for the zeta-functions, as well as to several closely related 
results. The notes concerning Ser's and Hasse's representations are given as theorems, while the related 
expansions are given either as separate theorems or as formul\ae~inside the remarks and corollaries. 
In the first theorem, we show that the famous Hasse's series for the
zeta-function, obtained in 1930 and named after the German mathematician Helmut Hasse, is 
equivalent to an earlier expression given by a little-known French mathematician Joseph Ser in 1926.
In the second theorem, we derive a similar series representation for the zeta-function 
involving the Cauchy numbers of the second kind (N{\o}rlund numbers). 
In the third theorem, with the aid of some special polynomials,
we generalize the previous results to the Hurwitz zeta-function.
In the fourth theorem, we obtain a similar series with Gregory's coefficients of higher order, 
which may also be regarded as a functional equation for the zeta-functions. 
In the fifth theorem, we extend the results of the third theorem to a class of Dirichlet series. As a consequence, 
we obtain several globally convergent series for the zeta-functions.
They are complementary to Hasse's series, contain the same finite differences and also generalize Ser's results.
In the paper, we also show that Hasse's series may be obtained much more easily by using 
the theory of finite differences, and we demonstrate that there exist numerous series 
of the same nature. 
In the sixth theorem, we show that Hasse's series is a simple particular case of a more general class of series involving
the Stirling numbers of the first kind.
All the expansions derived in the paper lead, in turn, to the series expansions for the Stieltjes constants,
including new series with rational terms only for Euler's constant, for the Maclaurin
coefficients of the regularized Hurwitz zeta-function, for the logarithm of the gamma-function,
for the digamma and trigamma functions. 
Throughout the paper, we also mention several ``unpublished''
contributions of Charles Hermite, which were very close to the results of Hasse and Ser. 
Finally, in the Appendix, we prove an interesting integral representation
for the Bernoulli polynomials of the second kind,
formerly known as the Fontana--Bessel polynomials.
\pagestyle{myheadings}
\thispagestyle{empty}
\baselineskip=12.875pt

\vspace{-1em}

\section{Introduction}
The Euler--Riemann zeta-function
\be\notag
\,\zeta(s) \equiv\sum\limits_{n=1}^\infty n^{-s}
\,= \prod\limits_{n=1}^\infty \!\big(1-p_n^{-s} \big)^{-1} 
\,\,,\qquad
\begin{array}{l}
\Re s>1\, \\[1mm]
p_n\in\mathbbm{P}\,,
\end{array}
\ee
and its most common generalization, the Hurwitz zeta-function 
\be\notag
\,\zeta(s,v) \equiv\sum\limits_{n=0}^\infty (n+v)^{-s}\,,\qquad
\begin{array}{l}
\Re s>1\, \\[1mm]
v\in\mathbbm{C}\setminus\!\{0,-1,-2,\ldots\}\,,
\end{array}
\ee
$\zeta(s)=\zeta(s,1)$, 
are some of the most important special functions in analysis and number theory.
They were studied by many famous mathematicians, including Stirling, Euler, Malmsten,
Clausen, Kinkelin, Riemann, Hurwitz, Lerch, Landau, and continue to receive considerable attention from modern researchers.
In 1930, the German mathematician Helmut Hasse published a paper \cite{hasse_01}, in which he obtained and studied these 
globally convergent series for the $\zeta$--functions
\begin{eqnarray}
&&\zeta(s)\,=\,\frac{1}{\,s-1\,}\sum_{n=0}^\infty\frac{1}{\,n+1\,}\sum
_{k=0}^n (-1)^k \binom{n}{k}(k+1)^{1-s} \,,\label{ug87g} \\[3mm]
&&\zeta(s,v)\,=\,\frac{1}{\,s-1\,}\sum_{n=0}^\infty\frac{1}{\,n+1\,}\sum
_{k=0}^n (-1)^k \binom{n}{k}(k+v)^{1-s} \,, \label{jnd2h93ndd}
\end{eqnarray}
containing finite differences $\,\Delta^n 1^{1-s}\,$
and $\,\Delta^n v^{1-s}\,$ respectively.\footnote{For the definition of the finite difference operator, 
see \eqref{0i3u40jfmnr}.}~Hasse also remarked\footnote{Strictly speaking, the remark
was communicated to him by Konrad Knopp.} that the first series is quite similar to 
the Euler transformation of the $\eta$-function series
$\,\eta(s)\equiv\sum\limits_{n=1}^\infty (-1)^{n+1} n^{-s}
=\big(1-2^{1-s}\big) \,\zeta(s)\,$, $\,\Re s>0\,$, i.e.
\vspace{-0.8em}
\be\label{wsdv34d}
\zeta(s)\,=\,\frac{1}{\,1-2^{1-s}\,}\sum_{n=0}^\infty\frac{1}{\,2^{n+1}\,}\sum
_{k=0}^n (-1)^k \binom{n}{k}(k+1)^{-s}  \,,
\ee
the expression which, by the way, was earlier
given by Mathias Lerch \cite{lerch_03}\footnote{Interestingly, Lerch did not notice that his result is simply
Euler's transformation of $\eta(s)$.} and later rediscovered by Jonathan Sondow \cite{sondow_03}, \cite{connon_06}.
Formul\ae~\eqref{ug87g}--\eqref{jnd2h93ndd} have become widely known, and in
the literature they are often referred to as \emph{Hasse's formul\ae} for the $\zeta$-functions.
At the same time, it is not so well known that 4 years earlier, a little-known French mathematician Joseph Ser 
published a paper \cite{ser_01} containing very similar results.\footnote{Moreover, as we come to show in Remark 2, 
series \eqref{jnd2h93ndd} is also a particular case of a more general formula obtained by N\o{}rlund in 1923
(although the formula may, of course, be much older).} In particular, he showed that
\begin{eqnarray}
\zeta(s)\!\! &&=\,\frac{1}{\,s-1\,}\sum_{n=0}^\infty\frac{1}{\,n+2\,}
\sum_{k=0}^{n} (-1)^k \binom{n}{k}(k+1)^{-s}, \label{hdf398hd}
\end{eqnarray}
and also gave this curious series
\begin{eqnarray}
\zeta(s)\!\! && =\,\frac{1}{\,s-1\,}
+\sum_{n=0}^\infty\big| G_{n+1}\big
|\sum_{k=0}^{n} (-1)^k \binom{n}{k}(k+1)^{-s} \,=\label{c289hdnwe3} \\[3mm]
&& =\,\frac{1}{\,s-1\,}+\frac{1}{2}
+ \frac{1}{12}\big(1-2^{-s}\big)+\frac{1}{24}\big(1-2\cdot 2^{-s}+3^{-s}\big) +\ldots \notag
\end{eqnarray}
\cite[Eq.~(4), p.~1076]{ser_01}\footnote{Our formula \eqref{c289hdnwe3} is a corrected version
of the original Ser's formula (4) \cite[p.~1076]{ser_01} (we also corrected this formula in our 
previous work \cite[p.~382]{iaroslav_09}). In Ser's formula (4), the corrections
which need to be done are the following. First, in the second line of (2)
the last term should be $(-1)^n (n+1)^{-s}$ and not $(-1)^n n^{-s}$. 
Second, in equation (3), ``(1-x((2-x)'' should read ``(1-x)(2-x)''.
Third, the region of convergence of formula (4), p.~1076, should be $s\in\mathbbm{C}\setminus\!\{1\}$
and not $s<1$.
Note also that Ser's $p_{n+1}$ are equal to our $|G_n|$. 
We,
actually, carefully \label{kj3re309ej}
examined 5 different hard copies of \cite{ser_01} (from the \emph{Institut Henri Poincaré}, from 
the \emph{École Normale Supérieure Paris}, from the \emph{Université Pierre-et-Marie-Curie}, 
from the \emph{Université de Strasbourg} and from the \emph{Bibliothèque nationale de France}), and all of them
contained the same misprints. }, \cite[p.~382]{iaroslav_09},
to which Charles Hermite was also very close already in 1900.\footnote{Charles Hermite even aimed to obtain 
a more general expression (see, for more details, Theorem \ref{ij298034hd} and footnote \ref{h92h3dbi23}).}
The numbers $G_n$ appearing in the latter expansion are known as \emph{Gregory's
coefficients} and may also be called by some authors \emph{(reciprocal) logarithmic numbers}, \emph{Bernoulli numbers of
the second kind}, normalized \emph{generalized Bernoulli numbers} $B_n^{(n-1)}$ and $B_n^{(n)}(1)$,\footnote{See \eqref{j023jd3ndu} hereafter
and also \cite[pp.~145--147, 462]{norlund_02}, \cite[pp.~127, 129, 135, 182]{milne_01}. Note also that $B_n^{(n-1)}=B_n^{(n-1)}(0)$ from
\eqref{j023jd3ndu}.} and normalized \emph{Cauchy numbers
of the first kind} $C_{1,n}$. They are rational and may be defined either via their generating function
\be\label{eq32}
\frac{z}{\ln(1+z)} \,=\, 1\,+\sum
_{n=1}^\infty G_n \, z^n ,\qquad|z|<1\,,
\ee
or explicitly via the recurrence relation 
\be\label{jfhek4s}
G_n\,=\,\frac{(-1)^{n-1}}{\,n+1\,}\,+ \sum_{k=1}^{n-1}\! \frac{(-1)^{n+1-k}G_k}{\,n+1-k\,}\,,
\qquad G_1=\frac{1}{2}\,,\quad n=2,3,4,\ldots
\ee
see e.g.~\cite[p.~10]{soldner_01}, \cite{bessel_01}, \cite[pp.~261--262, \S~61]{demorgan_01}, 
\cite[p.~143]{kluyver_02}, \cite[p.~423, Eq.~(30)]{kowalenko_01}, or via 
\begin{eqnarray}
\label{ldhd9ehn}
& G_n & =\,\frac{C_{1,n}}{n!}\,=-\frac{B_n^{(n-1)}}{\,(n-1)\,n!\,}=\frac{B^{(n)}_n(1)}{n!}\,
=  \int\limits_0^1\!\! \binom{x}{n} \,  dx\,\notag\\[-1pt]
&& = \, \frac{1}{n!}\! \int\limits_0^1\! (x-n+1)_n\,  dx\, = 
\, \frac{1}{n!}\!\sum\limits_{l=1}^{n} \frac{S_1(n,l)}{l+1} \, ,
\end{eqnarray}
where $n$ is a natural number, $\,(x)_n\equiv\,x\,(x+1)\,(x+2)\cdots(x+n-1)\,$ stands
for the Pochhammer symbol (also known as the rising factorial) and $S_1(n,l)$ are 
the Stirling numbers of the first kind, which are the coefficients in the expansion
of the falling factorial, and hence of the binomial coeffcient,\footnote{For more information about $S_1(n,l)$,
see \cite[Sect.~2]{iaroslav_08} and numerous references given therein.
Here we use exactly the same definition for $S_1(n,l)$ as in the cited reference.}
\be\label{09i4rjo2ende3}
(x-n+1)_n=\,
n!\binom{x}{n}=\,x\,(x-1)\,(x-2)\cdots(x-n+1)\,=\,\sum_{l=1}^n S_1(n,l)\, x^l .
\ee
Gregory's coefficients are alternating $\,G_n=(-1)^{n-1}|G_n|\,$ and decreasing in absolute value;
they behave as $\,\big(n\ln^2 n\big)^{-1}$ at $n\to\infty$ and may be bounded from below and from above accordingly
to formul\ae~(55)--(56) from \cite{iaroslav_08}.\footnote{For the asymptotics of $\,G_n\,$ and their history, 
see \cite[Sect.~3]{iaroslav_11}.} The first few coefficients are: 
$G_1=+\nicefrac{1}{2}\,$,
\mbox{$G_2=\nicefrac{-1}{12}\,$}, $G_3=\nicefrac{+1}{24}\,$, $G_4=\nicefrac{-19}{720}\,$,
$G_5=\nicefrac{+3}{160}\,$, $G_6=\nicefrac{-863}{60\,480}\,$,\ldots\footnote{Numerators and denominators of $G_n$
may also be found in OEIS A002206 and A002207 respectively.}
For more information about these important numbers, see \cite[pp.~410--415]{iaroslav_08}, \cite[p.~379]{iaroslav_09},
and the literature given therein (nearly 50 references).

\section{On the equivalence between Ser's and Hasse's representations for the Euler--Riemann zeta-function}
One may immediately see that Hasse's representation \eqref{ug87g} and Ser's  representation \eqref{hdf398hd} are very similar,
so one may question whether these expressions are equivalent or not.
The paper written by Ser \cite{ser_01} is much less cited than that by Hasse \cite{hasse_01}, and in the few works 
in which both of them are cited, these series are treated as different and with no connection between them.\footnote{A simple internet 
search with ``Google scholar'' indicates
that Hasse's paper \cite{hasse_01} is cited more than 50 times, while 
Ser's paper \cite{ser_01} is cited only 12 times (including several on-line resources such as \cite{weisstein_04}, as well as some incorrect items),
and all the citations are very recent. 
These citations mainly regard an infinite product for $e^{\gamma}$, e.g.~\cite{sondow_04}, 
\cite{bolbachan_01}, \cite{chen_02}, \cite{chen_03}, and we found only 
two works \cite{connon_06}, \cite{ahmed_02}, where both articles \cite{ser_01} and  \cite{hasse_01} 
were cited simultaneously and in the context of 
series representations for $\zeta(s)$.} However, as we shall show later, this is not true.

In one of our previous works \cite[p.~382]{iaroslav_09}, we already noticed that these two series are, in fact, equivalent, but this was stated in 
a footnote and without a proof.\footnote{We, however, indicated that a recurrence relation
for the binomial coefficients should be used for the proof.} 
Below, we provide a rigorous proof of this statement.
\begin{theorem}
Ser's representation for the $\zeta$-function \cite[p.~1076, Eq.~(7)]{ser_01}
\be\label{jhd38hdsa}
\zeta(s)\,=\,\frac{1}{\,s-1\,}\sum_{n=0}^\infty\frac{1}{\,n+2\,}
\sum_{k=0}^{n} (-1)^k \binom{n}{k}(k+1)^{-s}\,,\qquad s\in\mathbbm{C}\setminus\!\{1\}\,,
\ee
and Hasse's representation for the $\zeta$-function \cite[pp.~460--461]{hasse_01}
\be\label{jncw9eucn}
\zeta(s)\,=\,\frac{1}{\,s-1\,}\sum_{n=0}^\infty\frac{1}{\,n+1\,}\sum
_{k=0}^n (-1)^k \binom{n}{k}(k+1)^{1-s}\,,\qquad s\in\mathbbm{C}\setminus\!\{1\}\,,
\ee
are equivalent in the sense that
one series is a rearranged version of the other. 
\end{theorem}
\begin{proof}
In view of the fact that 
\be\notag
\frac{1}{\,k+1\,}\binom{n}{k}=\frac{1}{\,n+1\,}\binom{n+1}{k+1}\quad\text{and that}\quad
\frac{1}{\,(n+2)(n+1)\,}=\frac{1}{\,n+1\,} - \frac{1}{\,n+2\,} \, ,
\ee
Ser's formula \eqref{jhd38hdsa}, multiplied by the factor $s-1$, may be written as: 
\begin{eqnarray}
\left(s-1\right)\zeta(s)\!\!\!&&\,=\sum_{n=0}^\infty\frac{1}{\,n+2\,}
\sum_{k=0}^{n} (-1)^k \binom{n}{k}(k+1)^{-s}    \notag\\[1mm]
&&=\sum_{n=0}^\infty\frac{1}{\,(n+2)(n+1)\,}
\sum_{k=0}^{n} (-1)^k \binom{n+1}{k+1}(k+1)^{1-s} \notag\\[1mm]
&&=\sum_{n=0}^\infty\frac{1}{\,n+1\,}\sum_{k=0}^{n} (-1)^k \binom{n+1}{k+1}(k+1)^{1-s} -  \notag\\[1mm]
&&\qquad\qquad
-\sum_{n=0}^\infty\frac{1}{\,n+2\,}\sum_{k=0}^{n} (-1)^k \binom{n+1}{k+1}(k+1)^{1-s} = \notag\\[1mm] 
&&=1+\sum_{n=1}^\infty\frac{1}{\,n+1\,}\sum_{k=0}^{n} (-1)^k \binom{n+1}{k+1}(k+1)^{1-s} - \notag\\[1mm]
&&\qquad\qquad
-\sum_{n=1}^\infty\frac{1}{\,n+1\,}\sum_{k=0}^{n-1} (-1)^k \binom{n}{k+1}(k+1)^{1-s} =\notag\\[1mm]
&&=1+\sum_{n=1}^\infty\frac{(-1)^n}{\,(n+1)^s\,}
+\sum_{n=1}^\infty\frac{1}{\,n+1\,}\sum_{k=0}^{n-1} \frac{(-1)^k}{(k+1)^{s-1}}
\left\{ \!\binom{n+1}{k+1} - \binom{n}{k+1} \!\right\}\notag
\end{eqnarray}

\begin{eqnarray}
&&\,=\,1+\sum_{n=1}^\infty\frac{(-1)^n}{\,(n+1)^s\,}
+\,\sum_{n=1}^\infty\frac{1}{\,n+1\,}\sum_{k=0}^{n-1} \frac{(-1)^k}{(k+1)^{s-1}} \binom{n}{k}\notag\\[3mm] 
&&\,=\,1+
\sum_{n=1}^\infty\frac{1}{\,n+1\,}\sum_{k=0}^{n} \frac{(-1)^k}{(k+1)^{s-1}} \binom{n}{k}\,=
\sum_{n=0}^\infty\frac{1}{\,n+1\,}\sum_{k=0}^{n} \frac{(-1)^k}{(k+1)^{s-1}} \binom{n}{k}\,,\notag
\end{eqnarray}
where, between the seventh and eighth lines, we resorted to the recurrence relation for the binomial coefficients.
The last line is identical with Hasse's formula \eqref{jncw9eucn}. 
\end{proof}

Thus, series \eqref{jhd38hdsa} and \eqref{jncw9eucn} are equivalent in the sense that \eqref{jncw9eucn} 
may be obtained by an appropriate 
rearrangement of terms in \eqref{jhd38hdsa} and vice versa.
It is also interesting that Hasse's series for $\zeta(s)$ may be obtained from that for $\zeta(s,v)$ by 
simply setting $v=1$, while Ser's series for $\zeta(s)$, as we come to see later, 
is obtained from a much more complicated formula for $\zeta(s,v)$,
see \eqref{jhdf29h3r2}--\eqref{gh87tgkgft} hereafter.

\begin{corollary}\label{ihy923hbs}
The Stieltjes constants $\gamma_m$ may be given by the following series
\begin{eqnarray}
&&\displaystyle\gamma_m\,=\,-\frac{1}{\,m+1\,}\sum_{n=0}^\infty\frac{1}{\,n+2\,
}\sum_{k=0}^{n} (-1)^k \binom{n}{k}\frac{\ln^{m+1}(k+1)}{k+1}\,,
\qquad  m\in\mathbbm{N}_0\,, \label{oi0udhnnee}\\[3mm]
&&\displaystyle\gamma_m\,=\,-\frac{1}{\,m+1\,}\sum_{n=0}^\infty\frac{1}{\,n+1\,
}\sum_{k=0}^n (-1)^k \binom{n}{k}\ln^{m+1}(k+1)\,,
\qquad m\in\mathbbm{N}_0\,, \label{c34fc34tv}\\[3mm]
&&\displaystyle\gamma_m\,=\sum_{n=0}^\infty\big| G_{n+1}\big|\sum_{k=0}^{n}
(-1)^k \binom{n}{k}\frac{\ln^{m}(k+1)}{k+1}\,,
\qquad m\in\mathbbm{N}\,.\label{x23dcrcf}
\end{eqnarray}
\end{corollary}

\begin{proof}
The Stieltjes constants $\gamma_m\,$, $m\in\mathbbm{N}_0$, 
are the coefficients appearing in the regular part of the Laurent series expansion 
of $\zeta(s)$ about its unique pole $s=1$
\begin{eqnarray}
\label{dhd73vj6s1}
\zeta(s)\,=\,\frac{1}{\,s-1\,} + \gamma+\sum_{m=1}^\infty\frac{(-1)^m\gamma_m
}{m!}\,(s-1)^m \,,
\qquad s\in\mathbbm{C}\setminus\!\{1\}\,,
\end{eqnarray}
and $\gamma_0= \gamma$.\footnote{For more information on $\gamma_m\,$, 
see \cite[p.~166 \emph{et seq.}]{finch_01}, \cite{iaroslav_08}, \cite{iaroslav_07}, and the literature given therein.} 
Since the function $\,(s-1)\zeta(s)\,$ is holomorphic on the entire complex $s$--plane,
it may be expanded into the Taylor series. The latter expansion, applied to \eqref{jhd38hdsa} and 
\eqref{jncw9eucn} in a neighborhood of $s=1$, produces formul\ae~\eqref{oi0udhnnee}--\eqref{c34fc34tv}.
Proceeding similarly with $\,\zeta(s)-(s-1)^{-1}$ and Ser's formula \eqref{c289hdnwe3} yields \eqref{x23dcrcf}.
\end{proof}

\begin{corollary}\label{dc243d0jhxhdf}
The normalized Maclaurin coefficients $\delta_m$ of the regular function $\,\zeta(s)-(s-1)^{-1}\,$
admit the following series representation
\begin{eqnarray}\label{d23d2rgrrew}
\delta_m \,=\sum_{n=0}^\infty\big| G_{n+1}\big|\sum_{k=0}^{n}
(-1)^k \binom{n}{k}\ln^{m}(k+1)\,,
\qquad m\in\mathbbm{N}\,.
\end{eqnarray}
\end{corollary}

\begin{proof}
Analogously to the Stieltjes constants $\gamma_m$, may be introduced the normalized 
Maclaurin coefficients $\delta_m$ of the regular function $\,\zeta(s)-(s-1)^{-1}\,$
\be
\label{uyh0hf5}
\zeta(s)\,=\,\frac{1}{\,s-1\,} + \frac{1}{2} + \sum_{m=1}^\infty\frac{(-1)^m\delta_m
}{m!}\,s^m \,,
\qquad s\in\mathbbm{C}\setminus\!\{1\}.
\ee
They are important in the study of the $\zeta$-function 
and are polynomially related to the Stieltjes constants: $\,\delta_1=\frac12\ln2\pi-1\,$,
$\,\delta_2=\gamma_1+\frac12\gamma^2-\frac12\ln^2\! 2\pi-\frac{1}{24}\pi^2+2\,$, 
\ldots\footnote{For more information on $\delta_m\,$, see e.g.~\cite{lehmer_02}, 
\cite{sitaramachandrarao_01}, \cite{connon_08}, \cite[p.~168 \emph{et seq.}]{finch_01}.}
Now, from Ser's formula \eqref{c289hdnwe3}, it follows that 
\be\notag
\sum_{n=0}^\infty\big| G_{n+1}\big
|\sum_{k=0}^{n} (-1)^k \binom{n}{k}(k+1)^{-s}\,=\,\frac{1}{2} + \sum_{m=1}^\infty\frac{(-1)^m\delta_m
}{m!}\,s^m \,.
\ee
Expanding the left-hand side into the Maclaurin series and equating the coefficients of $s^m$ yield \eqref{d23d2rgrrew}.
\end{proof}

\section{A series for the zeta-function with the Cauchy numbers of the second kind (N{\o}rlund numbers)}
An appropriate rearrangement of terms in another series of Ser, formula \eqref{c289hdnwe3}, also leads to an interesting result. 
In particular, we may deduce a series very similar to \eqref{c289hdnwe3}, but containing
the normalized Cauchy numbers of the second kind $C_n$ instead of $G_n$.

The normalized Cauchy numbers of the second kind $C_n$, related to the ordinary Cauchy numbers of the
second kind $C_{2,n}$ as $C_n\equiv C_{2,n}/n!$ (numbers $C_{2,n}$ are also known as signless \emph{generalized Bernoulli numbers}
and signless \emph{N{\o}rlund numbers} $\big|B_n^{(n)}\big|=(-1)^n B_n^{(n)}$ ), appear in the power series expansion
of $\,\big[(1\pm z)\ln(1\pm z)\big]^{-1}$ and of $\ln\ln(1\pm z)$ in a neighborhood of zero
\be
\begin{cases}
\displaystyle
\frac{z}{(1+z)\ln(1+z)}\,=\,1+\sum_{n=1}^\infty\!C_n\, (-z)^n\, ,\quad
\qquad|z|<1\,,  \\[6mm]
\displaystyle
\ln\ln(1+z)\,=\,\ln z
+\sum_{n=1}^\infty\!\frac{C_n}{n} \, (-z)^n\, ,\quad
\qquad|z|<1\,,
\end{cases}
\ee
and may also be defined explicitly
\be
\begin{array}{ll}
\displaystyle C_n\,\equiv\,\frac{C_{2,n}}{n!}\, & \displaystyle =\,\frac{\,(-1)^n B_n^{(n)}\,}{n!}\, = \,
(-1)^n\!\!\int\limits_{0}^1\!\! \binom{x-1}{n} \, dx \,= \\[4mm]
\displaystyle
&\displaystyle  =\,
\frac{1}{n!}\!\int\limits_0^1\! (x)_n\, dx \,
=\,\frac{1}{n!}\!
\sum\limits_{l=1}^{n} \frac{|S_1(n,l)|}{l+1} \,
\,,\qquad n\in\mathbbm{N}\,.
\end{array}
\ee
They also are linked to Gregory's coefficients via the recurrence relation $\,C_{n-1}-C_{n}=|G_{n}|\,$, 
which is sometimes written as
\be
C_n \,=\,1-\sum_{k=1}^{n}\big|\,G_k\big| \,,\qquad n\in\mathbbm{N}\,.\quad\footnotemark
\ee
\footnotetext{This formula incorrectly appears in \cite[p.~267, 269]{jordan_01}:
in several places, the formul\ae~for $\psi_n(-1)$ should contain the sum with the upper bound $n$ instead of $n+1$.}
The numbers $C_n$ are positive rational and always decrease with $n$;
they behave as $\,\ln^{-1}n\,$ at $n\to\infty$ and may be bounded from below and from above accordingly
to formul\ae~(53)--(54) from \cite{iaroslav_08}. The first few values are: 
$C_{1}=\nicefrac{1}{2}\,$, \label{lkjfc02kme} $C_{2}=\nicefrac{5}{12}\,$,
$C_{3}=\nicefrac{3}{8}\,$, $C_{4}=\nicefrac{251}{720}\,$,
$C_{5}=\nicefrac{95}{288}\,$, $C_{6}=\nicefrac{19\,087}{60\,480}\,$, $\ldots$\footnote{See also OEIS A002657 and A002790, which
are the numerators and denominators respectively of $C_{2,n}=n!\,C_{n}$.}
For more information on the Cauchy numbers of the second kind, see \cite[pp.~150--151]{norlund_02}, \cite[p.~12]{davis_02}, 
\cite{norlund_01}, \cite[pp.~127--136]{milne_01}, \cite[vol.~III, pp.~257--259]{bateman_01},
\cite[pp.~293--294, \no13]{comtet_01}, \cite{howard_03,adelberg_01,zhao_01,qi_01}, 
\cite[pp.~406, 410, 414--415, 428--430]{iaroslav_08}, \cite{iaroslav_09}.

\begin{theorem}\label{h293hs}
The $\zeta$-function may be represented by the following globally convergent series
\begin{eqnarray}
\zeta(s)\!&&
=\,\frac{1}{\,s-1\,}+1-\sum_{n=0}^\infty C_{n+1}
\sum_{k=0}^{n} (-1)^k \binom{n}{k}(k+2)^{-s}\, =\label{jdf9243dn2} \\[3mm]
&&
=\,\frac{1}{\,s-1\,} + 1 - 2^{-s-1} - \frac{5}{12}\big(2^{-s}-3^{-s}\big) 
- \frac{3}{8}\left(2^{-s}-2\cdot3^{-s}+4^{-s}\right)-\ldots\notag
\end{eqnarray}
$s\in\mathbbm{C}\setminus\!\{1\}$,
where $C_n$ are the normalized Cauchy numbers of the second kind. 
\end{theorem}
\begin{proof}
Using Fontana's identity $\sum|G_n|=1\,$,  where the summation extends over 
positive integers $n$, see e.g.~\cite[p.~410, Eq.~(20)]{iaroslav_08}, Ser's formula \eqref{c289hdnwe3} takes the form
\begin{eqnarray}
\zeta(s)\, =\,\frac{1}{\,s-1\,}+1+\sum_{n=1}^\infty\big| G_{n+1}\big
|\sum_{k=1}^{n} (-1)^k \binom{n}{k}(k+1)^{-s} \,.\label{j932deed} 
\end{eqnarray}
Now, by taking into account that \mbox{$\,C_{n-1}-C_{n}=|G_{n}|\,$}, and by employing the recurrence
property of the binomial coefficients
\be\notag
\binom{n+1}{k}-\binom{n}{k} = \binom{n}{k-1}\,,
\ee
we find that  
\begin{eqnarray}
&&\hspace{-33mm}\zeta(s)-\frac{1}{\,s-1\,}-1\,=   \notag\\
\notag
\end{eqnarray}

\begin{eqnarray}
&&\quad = 
\sum_{n=1}^\infty C_n\sum_{k=1}^{n} (-1)^k \binom{n}{k}(k+1)^{-s} 
- \sum_{n=1}^\infty C_{n+1}\sum_{k=1}^{n} (-1)^k \binom{n}{k}(k+1)^{-s}\notag\\[3mm]
&&\quad=\sum_{n=0}^\infty C_{n+1}\sum_{k=1}^{n+1} (-1)^k \binom{n+1}{k}(k+1)^{-s} 
- \sum_{n=1}^\infty C_{n+1}\sum_{k=1}^{n} (-1)^k \binom{n}{k}(k+1)^{-s}\notag\\[3mm] 
&&\quad= \,-C_1\, 2^{-s} +\sum_{n=1}^\infty C_{n+1}\left\{\sum_{k=1}^{n+1} (-1)^k \binom{n+1}{k}(k+1)^{-s} 
- \sum_{k=1}^{n} (-1)^k \binom{n}{k}(k+1)^{-s}\right\}\notag\\[3mm] 
&&\quad= \, - C_1\,2^{-s}+\sum_{n=1}^\infty C_{n+1}\left\{\frac{(-1)^{n+1}}{(n+2)^{s}} 
+ \sum_{k=1}^{n} (-1)^k \binom{n}{k-1}(k+1)^{-s}\right\}\notag\\[3mm] 
&&\quad= - C_1\,2^{-s}- \sum_{n=1}^\infty C_{n+1}
\sum_{k=0}^{n} (-1)^k \binom{n}{k}(k+2)^{-s} \,,
\end{eqnarray}
which is the same as \eqref{jdf9243dn2}, since $\,C_1\,2^{-s}$ is the zeroth term of the last sum.
The global convergence of \eqref{jdf9243dn2} follows from that of \eqref{c289hdnwe3}. 
\end{proof}

\begin{corollary}
The Stieltjes constants may be represented by the following series
\be
\gamma_m\,=\,-  \sum_{n=0}^\infty C_{n+1}\sum_{k=0}^{n} (-1)^k \binom{n}{k}\frac{\ln^m (k+2)}{k+2}
\,,\qquad m\in\mathbbm{N}_0\,.
\ee
\end{corollary}

\begin{proof}\label{oi34fhs}
Proceeding analogously to the proof of Corollary \ref{ihy923hbs}, we obtain the required result.
Moreover, for Euler's constant $\gamma=\gamma_0$, we have
an expression which may be simplified thanks to 
the fact $\left.(-1)^n\Delta^n x^{-1}\right|_{x=2}=\big[(n+1)\,(n+2)\big]^{-1}$, so that
\begin{eqnarray}
\gamma\!\!&&=\,1 -  \sum_{n=0}^\infty C_{n+1}
\sum_{k=0}^{n} \frac{(-1)^k}{k+2} \binom{n}{k}\,=
\,1 -  \sum_{n=0}^\infty \frac{C_{n+1}}{\,(n+1)\,(n+2)\,} \,=\notag\\[1mm] 
&&\,=\,
1-\frac{1}{4}-\frac{5}{72}-\frac{1}{32}-\frac{251}{14\,400}-\frac
{19}{1728} -
\frac{19\,087}{2\,540\,160} - \ldots\notag
\end{eqnarray}
This series was already encountered in earlier works 
\cite[p.~380, Eq.~(34)]{iaroslav_08}, \cite[p.~429, Eq.~(95)]{iaroslav_09}.
\end{proof}

\begin{corollary}
The coefficients $\delta_m$ admit the following series representation
\be
\delta_m\,=\,- \sum_{n=0}^\infty C_{n+1}\sum_{k=0}^{n} (-1)^k \binom{n}{k}\ln^m (k+2)
\,,\qquad m\in\mathbbm{N}\,.
\ee
\end{corollary}

\begin{proof}
Analogous to Corollary \ref{dc243d0jhxhdf}, except that 
we replace \eqref{c289hdnwe3} by \eqref{jdf9243dn2}.
\end{proof}

\section{Generalizations of series with Gregory's coefficients and Cauchy numbers of the second kind 
to the Hurwitz zeta-function and to some Dirichlet series}
\begin{theorem}\label{ij298034hd}
The Hurwitz zeta-function $\,\zeta(s,v)$ may be represented by the following globally
series $\,s\in\mathbbm{C}\setminus\!\{1\}$ with the finite difference $\Delta^n v^{-s}$
\begin{eqnarray}
\zeta(s,v) \!\! && =\,\frac{v^{1-s}}{\,s-1\,}+\sum_{n=0}^\infty\big| G_{n+1}\big
|\sum_{k=0}^{n} (-1)^k \binom{n}{k}(k+v)^{-s}\,=\,\frac{v^{1-s}}{\,s-1\,} + \frac{1}{2} v^{-s}+  \notag\\[3mm] 
&& \quad +\frac{1}{12}\big[v^{-s}-(1+v)^{-s}\big] + 
\frac{1}{24}\big[v^{-s}-2 (1+v)^{-s}+(2+v)^{-s}\big] + \ldots\,,\qquad \label{jxh293nbwes}
\end{eqnarray}
where $\,\Re v>0\,$, 
\begin{eqnarray}
\zeta(s,v)\!\! && =\,\frac{(v-1)^{1-s}}{\,s-1\,}-\sum_{n=0}^\infty C_{n+1}
\sum_{k=0}^{n} (-1)^k \binom{n}{k}(k+v)^{-s}\, =\,\frac{(v-1)^{1-s}}{\,s-1\,}  - \frac{1}{2} v^{-s} - \quad \notag\\[3mm] 
&& \quad - \frac{5}{12}\big[v^{-s}-(1+v)^{-s}\big]
- \frac{3}{8}\big[v^{-s}-2\,(1+v)^{-s}+(2+v)^{-s})\big]-  \ldots\,,\qquad \label{jx20nsw}
\end{eqnarray}
where $\,\Re v>1\,$, and
\be\label{98y6b796}
\zeta(s,v)= \,\frac{1}{m\,(s-1)}\!\sum_{n=0}^{m-1} (v+a+n)^{1-s}  
+  \frac{1}{m}\!\sum_{n=0}^\infty (-1)^n N_{n+1,m}(a)
\sum_{k=0}^{n} (-1)^k \binom{n}{k} (k+v)^{-s}\quad
\ee
where $\,\Re v>-\Re a\,$, $\,\Re a \geqslant -1$ and $N_{n,m}(a)$ are the polynomials defined by
\be\label{oihf93h4}
N_{n,m}(a)\,\equiv\,\frac{1}{\,n!\,}\!\!\! \int\limits_a^{a+m}\!\! (x-n+1)_n\,  dx\,
=\!\!\int\limits_a^{a+m}\!\!\! \binom{x}{n}\,  dx\,
= \psi_{n+1}(a+m) - \psi_{n+1}(a)\,, \qquad
\ee
or equivalently by their generating function
\be\label{90823u4h32}
\frac{(1+z)^{a+m}-(1+z)^{a}}{\ln(1+z)}\,=
\sum_{n=0}^\infty N_{n,m}(a) \, z^n\,, \qquad |z|<1\,.
\ee
The function $\psi_n(x)$ is the antiderivative of the binomial coefficient and is also known
as the \emph{Bernoulli polynomials of the second kind} and the \emph{Fontana--Bessel polynomials}. 
All these series are similar to Hasse's series \eqref{jnd2h93ndd} and contain
the same finite difference $\Delta^n v^{-s}$. 
\end{theorem}

\begin{proof}
\emph{First variant of proof of \eqref{jxh293nbwes}: }
Expanding the function $\zeta(s,x)$ into 
the Gregory--Newton interpolation series
(also known as the forward difference formula) in a neighborhood of $x=v$,
yields
\be\label{kj3094fn3er}
\zeta(s,x+v)\,=\,\zeta(s,v)+ \sum_{n=1}^\infty\! \binom{x}{n}\Delta^n \zeta(s,v)\,
=\,\zeta(s,v)+ \sum_{n=1}^\infty \frac{(x-n+1)_n}{n!} \,\Delta^n \zeta(s,v)\,,
\ee
where $\Delta^n f(v)$ is the $n$th \emph{finite forward difference} of $f(x)$ at point $v$
\begin{eqnarray}
\Delta f(v)\! &&\displaystyle\equiv\,\left.\vphantom{\ln^{m+1}}\Delta f(x)\right|_{x=v}
=\, \left.\vphantom{\ln^{m+1}} \right\{\! f(x+1) -  f(x) \! \left\}\vphantom{\ln^{m+1}}\right|_{x=v}=\,  f(v+1) -  f(v)\,,\notag\\
\ldots\ldots  && \notag\\[1mm]
\Delta^n f(v)\! &&\displaystyle
=\,\Delta^{n-1} f(v+1) - \Delta^{n-1} f(v)\,=\,\ldots \label{0i3u40jfmnr}\\[3mm]
&&\displaystyle\notag
\,\ldots\,= \sum\limits_{k=0}^n (-1)^k \binom{n}{k} f(n-k+v)\,
=\,(-1)^n\sum\limits_{k=0}^n (-1)^k \binom{n}{k} f(v+k)\qquad
\end {eqnarray}
with 
$\,\Delta^0 f(v)\equiv f(v)\,$ by convention\footnote{Note, however, that due to the fact that finite differences may be defined
in slightly different ways and that there also exist \emph{forward}, \emph{central}, \emph{backward}
and other finite differences, 
our definition for $\Delta^n f(v)$ may not be shared by others.
Thus, some authors call the quantity $(-1)^n\Delta^n f(v)$ the $n$th finite difference,
see e.g.~\cite[p.~270, Eq.~(14.17)]{vorobiev_01} (we also employed the latter definition in \cite[p.~413, Eq.~(39)]{iaroslav_08}).
For more details on the Gregory--Newton interpolation formula, 
see e.g.~\cite[\S~9$\cdot$02]{jeffreys_02}, \cite[pp.~57--59]{milne_01}, \cite[pp.~184, 219 \emph{et seq.}, 357]{jordan_01},
\cite[Ch.~III]{boole_01}, \cite[Ch.~1 \& 9]{hamming_01}, \cite{norlund_02},  \cite{whittaker_02}, \cite[Ch.~3]{krylov_01}, \cite[p.~192]{goldstine_01}, 
\cite[Ch.~V]{melentiev_01_eng},  \cite[Ch.~III, pp.~184--185]{kantorovich_01_eng}, \cite{milne_02}, \cite{slavic_02}, \cite[p.~31]{phillips_02}.}.
Since the operator of finite difference $\,\Delta^n\,$ is linear
and because $\,\zeta(s,v+1)\,=\,\zeta(s,v) - v^{-s}$,
it follows from \eqref{0i3u40jfmnr} that $\,\Delta^n \zeta(s,v) \,=\,-\Delta^{n-1}v^{-s} \,$.
Formula \eqref{kj3094fn3er}, therefore, becomes
\begin{eqnarray}\label{8934dh2nd4}
\zeta(s,v) \,=\, \sum_{n=0}^\infty (x+v+n)^{-s} 
+ \sum_{n=1}^\infty \frac{(x-n+1)_n}{n!} \, \Delta^{n-1} v^{-s} \,.
\end{eqnarray}
Integrating termwise the latter equality over $\,x\in[0,1]\,$ and accounting for the fact that 
\be\label{094ufn32o4nf}
\int\limits_0^1 \! (x+v)^{-s} dx \,=\, \frac{1}{s-1}\left\{v^{1-s} -  (v+1)^{1-s}\right\}\,,
\ee
as well as using \eqref{ldhd9ehn}, we have
\begin{eqnarray}
\zeta(s,v) \! &&\displaystyle  =\, \frac{1}{s-1}\left\{\sum_{n=0}^\infty (v+n)^{1-s} - \sum_{n=0}^\infty (v+1+n)^{1-s} \right\} 
+ \sum_{n=1}^\infty G_n\, \Delta^{n-1} v^{-s} \notag\\[3mm]
&&\displaystyle =\, \frac{v^{1-s}}{s-1}  +\sum_{n=0}^\infty G_{n+1}\, \Delta^{n} v^{-s} \,,\label{h928h34dbd}
\end{eqnarray}
which is identical with \eqref{jxh293nbwes}, because $\,G_{n+1}=(-1)^n|G_{n+1}|\,$ and
\be\label{98yrbfe45}
\Delta^{n}v^{-s} \,\equiv 
\left. \Delta^{n} x^{-s}\right|_{x=v}\,= \,(-1)^n\!\sum_{k=0}^{n} (-1)^k \binom{n}{k} (k+v)^{-s}\,.
\ee
The reader may also note that if we put $v=1$ in \eqref{8934dh2nd4} and 
\eqref{98yrbfe45}, then we obtain Ser's formul\ae~(3) and (2) respectively. \qed

\emph{Second variant of proof of \eqref{jxh293nbwes}: }
Consider the generating equation for the numbers $G_n$, formula \eqref{eq32}. 
Dividing it by $z$ and then putting $\,z=e^{-x}-1\,$ yields
\begin{eqnarray}
\frac{1}{\,1-e^{-x}\,}\!\!&&=\,\frac{1}{\,x\,} + \big|G_{1}\big| 
+\sum_{n=1}^\infty \big|G_{n+1}\big| \, (1-e^{-x})^{n} \label{hgc38ebd}\\[3mm]
&&=\,\frac{1}{\,x\,} + \big|G_{1}\big| 
+\sum_{n=1}^\infty \big|G_{n+1}\big|\sum_{k=0}^n (-1)^k \binom{n}{k} e^{-kx} \,,\qquad x>0\,,\notag
\end{eqnarray}
since $\,|G_{n+1}|=(-1)^{n}G_{n+1}\,$.
Now, using the well--known integral representation of the Hurwitz $\zeta$-function
\be\notag
\zeta(s,v)\,=\,\frac{1}{\,\Gamma(s)\,}\!\int\limits_{0}^\infty\! \frac{e^{-vx}\, x^{s-1}}{\,1-e^{-x}\,}\, dx
\ee
and Euler's formul\ae 
\be\notag
\frac{v^{1-s}}{\,s-1\,}\,=\,\frac{1}{\,\Gamma(s)\,}\!\int\limits_{0}^\infty \! e^{-vx}\, x^{s-2}\, dx\,,\qquad
\quad v^{-s}\,=\,\frac{1}{\,\Gamma(s)\,}\!\int\limits_{0}^\infty \! e^{-vx}\, x^{s-1}\, dx\,,
\ee
we obtain
\begin{eqnarray}
\zeta(s,v)-\frac{v^{1-s}}{\,s-1\,}- \big|G_{1}\big|\,v^{-s} \!\! &&=\,\frac{1}{\,\Gamma(s)\,}\!\int\limits_{0}^\infty\!
e^{-vx}\, x^{s-1} \left\{\frac{1}{\,1-e^{-x}\,}-\frac{1}{\,x\,}- \big|G_{1}\big|\right\} \, dx \notag\\[3mm]
&&
=\,\frac{1}{\,\Gamma(s)\,}\!\sum_{n=1}^\infty \big|G_{n+1}\big| \sum_{k=0}^n (-1)^k \binom{n}{k} \!
\int\limits_{0}^\infty\! e^{-(k+v)x}\, x^{s-1} \, dx \notag\\[3mm]
&&
=\sum_{n=1}^\infty \big|G_{n+1}\big| \sum_{k=0}^n (-1)^k \binom{n}{k} (k+v)^{-s}. \notag
\end{eqnarray}
Remarking that $\big|G_{1}\big|\,v^{-s}$ is actually the term corresponding to $n=0$ in the sum on the right
yields \eqref{jxh293nbwes}.\footnote{It seems appropriate to note here that Charles Hermite in 1900 \label{h92h3dbi23}
tried to use a similar method to derive a series with Gregory's coefficients
for $\zeta(s,v)$, but his attempt was not succesfull. 
A careful analysis of his derivations \cite[p.~69]{hermite_01}, \cite[vol.~IV, p.~540]{hermite_02}, reveals that Hermite's errors is due to the 
incorrect expansion of \mbox{$\big(1-e^{-x}\big)^{-1}$} into the series with $\,\omega_n\,$, 
which, in turn, led him to an incorrect formula for $R(a,s)\equiv\zeta(s,a)$.\footnotemark{}
These results were never published during Hermite's lifetime and 
appeared only in epistolary exchanges with the Italian mathematician Salvatore Pincherle, who published them in \cite{hermite_01}
several months after Hermite's death. Later, these letters were reprinted in \cite{hermite_02}.}
\footnotetext{On p.~69 in \cite{hermite_01} and p.~540 in \cite[vol.~IV]{hermite_02} in the expansion for \mbox{$\big(1-e^{-x}\big)^{-1}$} 
the term $\omega_1$ should be replaced by $\omega_2$ and $\omega_n$ by $\omega_{n+1}$. 
Note that Hermite's $\omega_n=|G_n|$.} \qed\\

\emph{First variant of proof of \eqref{jx20nsw}: }
Integrating term--by--term the right--hand side of \eqref{8934dh2nd4} over $\,x\in[-1,0]\,$ and remarking that 
\be
\frac{1}{n!}\!\int\limits_{-1}^0\! (x-n+1)_n\, dx\,=
\int\limits_{-1}^0\!\! \binom{x}{n} \, dx \,=\,
\int\limits_{0}^1\!\! \binom{x-1}{n} \, dx \,=\,
\,(-1)^nC_n\,,
\ee
we have
\begin{eqnarray}
\displaystyle \zeta(s,v)\! && =\, \frac{1}{s-1}\left\{\sum_{n=0}^\infty (v-1+n)^{1-s} - \sum_{n=0}^\infty (v+n)^{1-s} \right\} 
+ \sum_{n=1}^\infty (-1)^n C_n\, \Delta^{n-1} v^{-s} \notag\\[3mm]
&&\displaystyle =\, \frac{(v-1)^{1-s}}{s-1}  - \sum_{n=0}^\infty C_{n+1}\, (-1)^n\Delta^{n} v^{-s} \,,
\end{eqnarray}
which coincides with \eqref{jx20nsw} by virtue of \eqref{98yrbfe45}. \qed

\emph{Second variant of proof of \eqref{jx20nsw}: }
In order to obtain \eqref{jx20nsw}, we also may proceed analogously to the demonstration of Theorem \ref{h293hs},
in which we replace \eqref{c289hdnwe3} by \eqref{jxh293nbwes}. The unity appearing from
Fontana's series in \eqref{jdf9243dn2} becomes $v^{-s}$ and the term $(k+2)^{-s}$ 
becomes $(k+1+v)^{-s}$, that is to say
\begin{eqnarray}\label{gh65rf6}
\zeta(s,v)\, =\,\frac{v^{1-s}}{\,s-1\,}+v^{-s}-\sum_{n=0}^\infty C_{n+1}
\sum_{k=0}^{n} (-1)^k \binom{n}{k}(k+v+1)^{-s}\,,
\end{eqnarray}
$\Re v > 0.$ Using the recurrence relation $\,\zeta(s,v)=\zeta(s,v+1)+v^{-s}\,$ and rewriting the final result 
for $v$ instead of $v+1$, we immediately obtain \eqref{jx20nsw}. \qed\\

\emph{Proof of \eqref{98y6b796}: }
Our method of proof, which uses the Gregory--Newton interpolation formula, 
may be further generalized.
By introducing polynomials $N_{n,m}(a)$ accordingly to \eqref{oihf93h4}, and
by integrating \eqref{8934dh2nd4} over $\,x\in[a,a+m]\,$, we have
\be\label{8923dhn3}
\zeta(s,v) =\, \frac{1}{\,m\,(s-1)\,}\Big\{\zeta(s-1,v+a) - \zeta(s-1,v+a+m) \Big\} 
+  \frac{1}{m}\!\sum_{n=1}^\infty N_{n,m}(a)\, \Delta^{n-1} v^{-s} 
\ee
Simplifying the expression in curly brackets and reindexing the latter sum immediately yields
\be\label{0934uf023jhd}
\zeta(s,v) =\, \frac{1}{\,m\,(s-1)\,}\sum_{n=0}^{m-1} (v+a+n)^{1-s}  
+  \frac{1}{m}\!\sum_{n=0}^\infty N_{n+1,m}(a)\, \Delta^{n} v^{-s} ,
\ee
which is identical with \eqref{98y6b796}. Note that expansions \eqref{jxh293nbwes}--\eqref{jx20nsw}
are both particular cases of \eqref{98y6b796} at $m=1$. Formula \eqref{jxh293nbwes} is obtained by setting $a=0$, 
while \eqref{jx20nsw} corresponds to $a=-1$.
\end{proof}

\begin{corollary}
The Euler--Riemann $\zeta$-function admits the following general expansions
\be\label{u928y43dh23}
\zeta(s) =\, \frac{1}{\,m\,(s-1)\,}\sum_{n=1}^{m} (a+n)^{1-s}  
+  \frac{1}{m}\!\sum_{n=0}^\infty (-1)^n N_{n+1,m}(a)
\sum_{k=0}^{n} (-1)^k \binom{n}{k} (k+1)^{-s}
\ee
$a>-1$, $m\in\mathbbm{N}$, and 
\be\label{u928y43dh23b}
\zeta(s) =\,1\, + \frac{1}{\,m\,(s-1)\,}\sum_{n=1}^{m} (a+1+n)^{1-s}  
+  \frac{1}{m}\!\sum_{n=0}^\infty (-1)^n N_{n+1,m}(a)
\sum_{k=0}^{n} (-1)^k \binom{n}{k} (k+2)^{-s}
\ee
$a>-2$, $m\in\mathbbm{N}$, containing finite differences $\Delta^n 1^{-s}$ and 
$\Delta^n 2^{-s}$ respectively.
Ser's series \eqref{c289hdnwe3} and our series from Theorem \ref{h293hs} are particular cases
of the above expansions.\footnote{We have Ser's formula when putting $a=0$, $m=1$ in \eqref{u928y43dh23}, 
and our Theorem \ref{h293hs} if setting $a=-1$, $m=1$ in \eqref{u928y43dh23b}.}
\end{corollary}

\begin{proof}
On the one hand, setting $v=1$ in \eqref{0934uf023jhd} we immediately obtain \eqref{u928y43dh23}.
On the other hand, putting $v+1$ instead of $v$ and using the relation $\,\zeta(s,v+1)\,=\,\zeta(s,v) - v^{-s}$, 
equality \eqref{0934uf023jhd} takes the form
\be
\zeta(s,v) =\, v^{-s}+ \frac{1}{\,m\,(s-1)\,}\sum_{n=1}^{m} (v+a+n)^{1-s}  
+  \frac{1}{m}\!\sum_{n=0}^\infty N_{n+1,m}(a)\, \Delta^{n} (v+1)^{-s} .
\ee
At point $v=1$, this equality becomes \eqref{u928y43dh23b}. Continuing the process,
we may also obtain similar formul\ae~for $\zeta(s)$ containing finite differences $\Delta^n 3^{-s}$,
$\Delta^n 4^{-s}$,\ldots 
\end{proof}

\noindent{\bfseries Remark 1, related to the polynomials $\boldsymbol{N_{n,m}(a)}.\,$}\label{9874rhf32d}
Polynomials $N_{n,m}(a)$ generalize many special numbers and have a variety of interesting properties.
First of all, we remark that $N_{n,m}(a)$ are polynomials of degree $n$ in $a$ with rational coefficients.
This may be seen from the fact that
\begin{eqnarray}\label{oihf93h4b}
& \displaystyle N_{n,m}(a)\,\equiv\,\frac{1}{n!}\!\! \int\limits_a^{a+m}\!\! (x-n+1)_n\,  dx\,& \displaystyle =
\,\frac{1}{n!}\sum_{l=1}^n \frac{S_1(n,l)}{l+1}\Big\{(a+m)^{l+1} - a^{l+1} \Big\} \qquad \qquad\\[1mm]
 && \displaystyle =\,\frac{1}{n!}\sum_{l=1}^n \frac{S_1(n,l)}{l+1}\sum_{k=0}^l a^k m^{l+1-k}\binom{l+1}{k} \,,\notag
\end{eqnarray}
where $n$ and $m$ are positive integers.
This formula is quite simple and very handy for the calculation of $ N_{n,m}(a)$ with the help of CAS.
It is therefore clear that for any $a\in\mathbbm{Q}$, polynomials
$N_{n,m}(a)$ are simply rational numbers. Some of such examples may be of special interest
\be\notag
N_{n,1}(-1)\,=\,(-1)^n C_n\,,\qquad 
N_{n,1}(0)\,=\, G_n\,,\qquad 
N_{2n,1}(n-1)\,=\,M_{2n}\,,\qquad 
\ee
where $M_n$ are central difference coefficients $M_2=\nicefrac{-1}{12}\,$,
\mbox{$M_4=\nicefrac{+11}{720}\,$}, $M_6=\nicefrac{-191}{60\,480}\,$, $M_8=\nicefrac{+2497}{3\,628\,800}\,$,
\,\ldots\,, see e.g.~\cite{salzer_01}, \cite[\S~9$\cdot$084]{jeffreys_02}, \cite[p.~186]{milne_01}, OEIS A002195 and A002196.
The derivative of $N_{n,m}(a)$ is 
\be\label{fhqgg6754r}
\frac{\partial  N_{n,m}(a)}{\partial  a} \,=\frac{\,(a+m-n+1)_n - (a-n+1)_n}{n!}
\,=\binom{a+m}{n} - \binom{a}{n} \,.
\ee
Polynomials $N_{n,m}(a)$ are related to many other special polynomials. This can be readily seen from the generating
equation for $N_{n,m}(a)$, which we gave in \eqref{90823u4h32} without the proof. Let us now prove it.
On the one hand, on integrating $\,(1+z)^x\,$ between $x=a$ and $x=a+m$ we undoubtedly have
\be
\int\limits_a^{a+m}\!\!\! (1+z)^x\,  dx\,=\,\frac{(1+z)^{a+m}-(1+z)^{a}}{\ln(1+z)}\,.
\ee
On the other hand, the same integral may be calculated by expanding $\,(1+z)^x\,$ into the binomial series
\be
\int\limits_a^{a+m}\!\!\! (1+z)^x\,  dx\,= 
\sum_{n=0}^\infty z^n \!\!\!\! \int\limits_a^{a+m}\!\!\! \binom{x}{n}\,  dx\, =
\sum_{n=0}^\infty  z^n N_{n,m}(a)
\ee
by virtue of the uniform convergence.
Equating both expressions yields \eqref{90823u4h32}.
There is also a more direct way to prove the same result and, in addition, to explicitly determine $N_{0,m}(a)$.
Using \eqref{oihf93h4b} and accounting for the 
absolute convergence, we have for the right part of \eqref{90823u4h32}
\begin{eqnarray}
&\displaystyle\sum_{n=1}^\infty N_{n,m}(a) \, z^n\,=\,\sum_{l=1}^\infty \frac{(a+m)^{l+1} - a^{l+1}}{l+1}
\underbrace{\sum_{n=1}^\infty \frac{S_1(n,l)}{n!}\, z^n }_{\ln^l(1+z)/l!} \,=\,\frac{1}{\ln(1+z)}\cdot   \notag\\[-4mm]
&\displaystyle\label{8u4rfhq2me}\\
&\displaystyle  \cdot\left\{\sum_{l=2}^\infty \frac{\big[(a+m)\ln(1+z)\big]^l}{l!}
- \sum_{l=2}^\infty \frac{\big[a\ln(1+z)\big]^l}{l!}  \right\}= 
\,\frac{(1+z)^{a+m}-(1+z)^{a}}{\ln(1+z)}\,-\,m \,,   \notag
\end{eqnarray}
where we used the generating equation for the Stirling numbers of the first kind,
see e.g.~\cite[p.~408]{iaroslav_08}, \cite[p.~369]{iaroslav_09}. Hence, $N_{0,m}(a)=m$. 
Polynomials $N_{n,m}(a)$ are, therefore,  close
to the Stirling polynomials, to Van Veen's polynomials $K_n^{(z)}$ \cite{van_veen_01},
to various generalizations of the Bernoulli numbers/polynomials,
including the so--called \emph{N{\o}rlund polynomials} \cite[p.~602]{nemes_03}, which are also known 
as the generalized Bernoulli polynomials of the second kind \cite[p.~324, Eq.~(2.1)]{carlitz_01}, and to many other special polynomials,
see e.g.~\cite[Vol.~III, \S~19]{bateman_01}, \cite[Ch.~VI]{milne_01}, \cite[Vol.~I, \S~2.8]{luke_01},
\cite{carlitz_02}, \cite{norlund_02}, \cite{norlund_01}, \cite{gould_01},
\cite{gould_02}, \cite{young_02} \cite{brychkov_01}, \cite{brychkov_02}, \cite{roman_01},  \cite{kowalenko_01},
\cite{rubinstein_01}.
The most close connection seems to exist 
with the \emph{Bernoulli polynomials of the second kind}, also known as the \emph{Fontana--Bessel polynomials}, 
which are denoted by $\psi_n(x)$ by Jordan \cite{jordan_02}, 
\cite[Ch.~5]{jordan_01}, \cite[p.~324]{carlitz_01},\footnote{These
polynomials and/or those equivalent or closely related to them, were rediscovered in numerous works and by numerous authors
(compare, for example, works \cite[p.~1916]{merlini_01}, \cite[p.~3998]{young_03}, \cite[\S~5.3.2]{roman_01}, or
compare the Fontana--Bessel polynomials from \cite{appel_03}\footnotemark with 
$\psi_n(x)$ introduced by Jordan in the above--cited works and also with polynomials $P_{n+1}(y)$ employed by Coffey in \cite[p.~450]{coffey_02}), 
so we give here
only the most frequent notations and definitions for them.}
\footnotetext{Apparently in \cite{appel_03} Appell refers to Bessel's work \cite{bessel_01} of 1811, which is also indirectly mentioned 
by Vacca in \cite{vacca_02}. 
In \cite{bessel_01}, we find some elements related to the polynomials $\psi_{n}(x)$, as well as to  $N_{n,m}(a)$,
but Bessel did not perform their systematical study.
In contrast, he studied quite a lot a particular case of them leading  
to Gregory's coefficients $G_n=\psi_n(0)$, $A^{0}, A', A'',\ldots$ in Bessel's notations, see
\cite[pp.~10--11]{bessel_01}, \cite[p.~207]{vacca_02}. In this context, it may be interesting to remark
that Ser \cite{ser_02}, \cite{appel_04}, investigated polynomials $\psi_n(x)$ more in details [he denoted them $P_{n+1}(y)$], and this before Jordan, and also
calculated several series with $G_n$.}
and have the following generation function
\be\label{oi09u03dj}
\frac{\,z\,(z+1)^x}{\ln(z+1)}\,=\sum_{n=0}^\infty \psi_n(x) \, z^n\,,\qquad |z|<1\,,
\ee
see e.g.~\cite[Ch.~5, p.~279, Eq.~(8)]{jordan_01}, \cite[p.~324, Eq.~(1.11)]{carlitz_01},
and with the Bernoulli polynomials of higher order, usually denoted by $B_n^{(s)}(x)$, and defined via 
\be\label{j023jd3ndu}
\frac{\,z^s\,e^{xz}}{(e^z-1)^s}\,=\sum_{n=0}^\infty \frac{B_n^{(s)}(x)}{n!} \, z^n\,,\qquad |z|<2\pi\,,
\ee
see e.g.~\cite[pp.~127--135]{milne_01}, \cite[p.~323, Eq.~(1.4)]{carlitz_01}, \cite{norlund_03}, \cite[p.~145]{norlund_02},
\cite{norlund_01}, \cite[Vol.~III, \S~19.7]{bateman_01}.
Indeed, using formula \eqref{oi09u03dj}, we have for the left part of \eqref{90823u4h32}
\begin{eqnarray}
\,\frac{(1+z)^{a+m}-(1+z)^{a}}{\ln(1+z)}&&
=\,\sum_{n=0}^\infty \big\{\psi_n(a+m)- \psi_n(a)\big\}\, z^{n-1} \notag\\
&& = \,\sum_{n=0}^\infty \big\{\psi_{n+1}(a+m)- \psi_{n+1}(a)\big\} \,z^{n}  \notag
\end{eqnarray}
since $\,\psi_0(x)=1\,$.\footnotemark[25]~Comparing the latter expression to the right
part of  \eqref{90823u4h32} immediately yields
\be\label{89rh43ujfb4}
N_{n,m}(a)\,=\,\psi_{n+1}(a+m) - \psi_{n+1}(a)\,=\sum_{k=0}^{m-1} \psi_{n}(a+k) \,,\qquad n\in\mathbbm{N}_0\,,
\ee
since $\,\Delta\psi_{n+1}(x) =\psi_{n}(x)\,$, see e.g.~\cite[pp.~265, 268]{jordan_01}.
In particular, for $m=1$ we simply have
\be\label{897tv87tc}
N_{n,1}(a)\,=\,\psi_{n}(a)\,.
\ee
Another way to prove \eqref{89rh43ujfb4} is to recall that $\,\binom{x}{n}\,dx=d\psi_{n+1}(x)\,$, see e.g.~\cite[p.~130]{jordan_02},
\cite[p.~265]{jordan_01}.
Hence, the antiderivative of the falling factorial is, up to a function of $n$,
precisely the function $\psi_{n+1}(x)$.
By virtue of this important property, 
formula \eqref{89rh43ujfb4} follows immediately
from our definition of $N_{n,m}(a)$ given in the statement of the theorem. 
Furthermore, from \eqref{j023jd3ndu} it follows that $\,B_n^{(n)}(x+1)\,=\,n!\,\psi_n(x)\,$, see e.g.~\cite[p.~135]{milne_01},
 \cite[Eq.~(2.1) \& (2.11)]{carlitz_01}, \cite[p.~147]{norlund_02}, whence
\be
N_{n,m}(a)\,=\,\frac{1}{(n+1)!}\left\{ B_{n+1}^{(n+1)}(a+m+1) -  B_{n+1}^{(n+1)}(a+1)\right\}\,,\qquad n\in\mathbbm{N}_0\,.
\ee
This clearly displays a close connection between $N_{n,m}(a)$ and the Bernoulli polynomials of both varieties. 
The latter have been the object of much research by N{\o}rlund\footnote{Also written N\"orlund.}
\cite{norlund_03}, \cite{norlund_02}, \cite{norlund_01},
\cite[Ch.~VI]{milne_01}, Jordan \cite{jordan_02}, \cite[Ch.~5]{jordan_01}, Carlitz \cite{carlitz_01} and some other authors.
The polynomials $N_{n,m}(a)$ may also be given by the following integrals
\begin{eqnarray} 
& \displaystyle 
N_{n,m}(a) &=\,\frac{(-1)^{n}}{\pi} 
\!\int\limits_0^\infty \! \frac{\,\pi \cos\pi a - \sin\pi a \ln x\,}{\,(1+x)^{n+1}\,} 
\cdot\frac{\,\big[(-x)^m-1\big]\, x^a \,}{\,\ln^2 x +\pi^2\,}\,dx   \label{fmiuhgd94} \\[2mm] 
&& \displaystyle\notag
=\,\frac{(-1)^{n}}{\pi} 
\!\int\limits_0^\infty \! \frac{\,\pi \cos\pi a + \sin\pi a \ln x\,}{\,(1+x)^{n+1}\,} 
\cdot\frac{\,\big[(-1)^m-x^m\big]\, x^{n-m-a-1} \,}{\,\ln^2 x +\pi^2\,}\,dx\,,
\end{eqnarray} 
provided that $m$ and $n$ are positive integers, and $-1\leqslant a\leqslant n-1$ and $-1-a\leqslant m\leqslant n-a-1$.
These representations follow from \eqref{89rh43ujfb4} and this formula
\be\label{fmiuhgd94}
\psi_{n}(x)\,=\,\frac{(-1)^{n+1}}{\pi} 
\!\int\limits_0^\infty \! \frac{\,\pi \cos\pi x - \sin\pi x \ln z\,}{\,(1+z)^n\,} 
\cdot\frac{z^x \, dz}{\,\ln^2 z +\pi^2\,}
\,,\quad  -1\leqslant x\leqslant n-1\,,\quad
\ee
whose proof we put in the Appendix.
At this stage, it may be useful to provide explicit expressions for the first few polynomials $N_{n,m}(a)$
\begin{eqnarray}
&&\hspace{-1.5mm} N_{0,m}(a)\,=\,m   \notag\\[1mm]
&&\hspace{-1.5mm} N_{1,m}(a)\,=\,ma+ \tfrac12 m^2   \label{uyh293dd}\\[1mm]
&&\hspace{-1.5mm} N_{2,m}(a)\,=\,\tfrac12 ma^2 + \tfrac12 am^2 - \tfrac12 am + \tfrac16 m^3 - \tfrac14 m^2  \notag \\[1mm]
&&\hspace{-1.5mm} N_{3,m}(a)\,=\,\tfrac13 am + \tfrac16 m^2 - \tfrac12 a^2m - \tfrac12 am^2 -
\tfrac 16 m^3 + \tfrac16 a^3m + \tfrac14 a^2m^2 + \tfrac16 am^3 + \tfrac{1}{24}m^4    \notag
\end{eqnarray}
and so on. They may be obtained either directly from \eqref{oihf93h4b}, or from the similar formul\ae~for the
polynomials $\psi_n(x)$
\be\label{d230i9jf}
\psi_{n}(x)\, =
\,\frac{1}{(n-1)!}\sum_{l=0}^{n-1} \frac{S_1(n-1,l)}{l+1}\, x^{l+1} \,+\, G_{n}\,,\qquad n\in\mathbbm{N},\quad\footnotemark
\ee
where $\,S_1(0,0)=1\,$, and from \eqref{89rh43ujfb4}.\footnotetext{In particular, \label{kic2ju09hrrf}
$\,\psi_1(x)=x+\frac12\,$, $\,\psi_2(x)=\frac12x^2-\frac{1}{12}\,$, 
$\,\psi_3(x)=\frac16x^3-\frac14x^2+\frac{1}{24}\,$,  
$\,\psi_4(x)=\frac{1}{24}x^4-\frac16x^3+\frac16x^2 -\frac{19}{720}\,$, etc. 
see e.g.~\cite[p.~272]{jordan_01}. The value $\,\psi_0(x)=1\,$ cannot be computed with the help of this formula
and directly follows from the generating equation \eqref{oi09u03dj} (the limit of the left part when $z\to0$ tends
to $1$ independently of $x$). Note also that $\,\psi_n(0)=G_n\,$, except for $n=0$. Formula \eqref{d230i9jf} may also be found in \cite[p.~267]{jordan_01} with a little corection 
to take into account: the upper bound of the sum in (6) should read $n$ instead of $n+1$.}
Finally, we remark that the complete asymptotics of the polynomials $N_{n,m}(a)$ at large $n$ are given by
\be\label{lkjc023nf}
N_{n,m}(a)\,\sim\,\frac{(-1)^{n+1}}{\,\pi \, n^{a+1}}
\sum_{l=0}^{\infty} \frac{1}{\ln^{l+1}n}\cdot
\Big[\sin\pi a\cdot\Gamma(a+1)\Big]^{(l)}
\,,\qquad 
\begin{array}{l}
n\to\infty \\[1mm]
\Re a \geqslant -1\,
\end{array}\,,
\ee
where $^{(l)}$ stands for the $l$th derivative, and $m$ is a finite natural number.
In particular, retaining first two terms, we have
\begin{eqnarray}
\displaystyle N_{n,m}(a)\,\sim && \displaystyle \frac{(-1)^{n+1}}{\pi\,n^{a+1}\ln n}
\cdot\left\{\sin\pi a \cdot\Gamma(a+1) \,+ 
\vphantom{\frac{\pi\cos\pi a \cdot \Gamma(a+1) + \sin\pi a \cdot \Gamma(a+1) \cdot\Psi(a+1)}{\ln n}}\label{809q0g08g}
\right.\\[1mm]
&&\notag\displaystyle\qquad\qquad\qquad\quad
\left.+ \,\frac{\pi\cos\pi a \cdot \Gamma(a+1) + \sin\pi a \cdot \Gamma(a+1) \cdot\Psi(a+1)}{\ln n}\right\}
\end{eqnarray}
at $n\to\infty$.
Both results can be obtained without difficulty from the complete asymptotics of $\,B_{n}^{(n)}(x)\,$
given by N{\o}rlund \cite[p.~38]{norlund_01}.
Note that if $a\in\mathbbm{N}_0$, the first term of asymptotics \eqref{lkjc023nf}--\eqref{809q0g08g} vanishes,
and thus $N_{n,m}(a)$ decreases slightly faster. Remark also that making $a\to0$ and $a\to-1$ 
in \eqref{lkjc023nf}--\eqref{809q0g08g}, we find the asymptotics
of numbers $G_n$ and $C_n$ respectively.\footnote{In \cite[p.~414, Eq.~(51)]{iaroslav_08}, 
we obtained the complete asymptotics for $\,C_n=C_{2,n}/n!=\big|B^{(n)}_n\big|/n!\,$ at large $n$. However, 
it seems appropriate to notice that the equivalent result may be straightforwardly derived from 
N{\o}rlund's asymptotics of $\,B_{n}^{(n)}(x)\,$, since $B_{n}^{(n)}(0)=B^{(n)}_n$, see \cite[pp.~27, 38, 40]{norlund_01}.}

\begin{theorem}\label{45yvgt4rf45}
For any positive integer $k>1$, the $\zeta$-functions $\zeta(s,v)$ and $\zeta(s)$ obey the following functional 
relationships 
\begin{eqnarray}
&\displaystyle\zeta(s,v) \,
&\displaystyle= \,-\!\sum_{l=1}^{k-1} \! \frac{\,(k-l+1)_l\,}{(s-l)_l}\cdot\zeta(s-l,v)  \,+ 
\sum_{l=1}^{k} \frac{\,(k-l+1)_l\,}{(s-l)_l} \cdot v^{l-s}  \,+\,  \label{984rhc44}\\[2mm]
&&\displaystyle\qquad
 +  \,k\! \sum_{n=0}^\infty (-1)^n G_{n+1}^{(k)}\sum_{k=0}^{n} (-1)^k \binom{n}{k} (k+v)^{-s} \,,\notag
\end{eqnarray}
and
\begin{eqnarray}\notag
\zeta(s) \,=\,-\!\sum_{l=1}^{k-1} \! \frac{\,(k-l+1)_l\,}{(s-l)_l} \zeta(s-l)  \,+ 
\,\frac{k}{s-k} \,+\,k\! \sum_{n=0}^\infty (-1)^n G_{n+1}^{(k)}\!\sum_{k=0}^{n} 
(-1)^k \binom{n}{k} (k+1)^{-s}\qquad\\
\notag
\end{eqnarray}
respectively, where $ G_n^{(k)}$ are \emph{Gregory's coefficients of higher order} defined as
\be\label{utrdswf9}
G_n^{(k)}\,\equiv\frac{1}{\,n!\,}\!\sum_{l=1}^n\frac{S_1(n,l)}{l+k}\,.
\ee
In fact, we come to prove a more general result, of which the above formul\ae~are two particular cases. 
\end{theorem}

\begin{proof}
Let $\,\rho(x)\,$ be the normalized weight such that 
\be\notag
\int\limits_{a}^{a+m} \!\!\! \rho(x)\, dx\,=\,1\,, \quad
\text{and let denote}\quad
N_{n,m}^{(\rho)}(a)\,\equiv\,\frac{1}{\,n!\,}\!\!\!\int\limits_{a}^{a+m} \!\!  (x-n+1)_n\,\rho(x)\, dx\,.
\ee
Performing the same procedure as in the case of \eqref{98y6b796} and assuming the uniform convergence, we obtain
\be\label{j09384jhfd}
\zeta(s,v) \,=\sum_{n=0}^\infty \!\underbrace{\int\limits_{a}^{a+m}\!\!\!\frac{\rho(x)}{(v+x+n)^s}\, dx}_{ F_{n,m,a}[\rho(x)]}\,
+  \sum_{n=1}^\infty N_{n,m}^{(\rho)}(a)\, \Delta^{n-1} v^{-s} \,.
\ee
Albeit this generalization appears rather theoretical, it, however, may be useful if 
the functional $F_{n,m,a}[\rho(x)]$ admits a suitable closed--form and if the series $\sum F_{n,m,a}[\rho(x)]$ converges.
Thus, if we simply put $\rho(x)=1/m$, then we retrieve our formula \eqref{98y6b796}. If we put $\,\rho(x)=k\,x^{k-1}$,
where $k\in\mathbbm{N}$, and set $a=0$, $m=1$, then it is not difficult to see that
\be\label{9832h3n4}
N_{n,1}^{(kx^{k-1})}(0)\,=\, \frac{k}{\,n!\,}\!\sum_{l=1}^n\frac{S_1(n,l)}{l+k}\,\equiv \,k\, G_n^{(k)}\,,
\ee
Now, remarking that the repeated integration by parts yields
\begin{eqnarray}
&\displaystyle\int \! x^{k-1} (v+x+n)^{-s}\, dx\,&\displaystyle
=\,\frac{1}{\,k\,} \! \sum_{l=1}^{k-1} \frac{(-1)^{l+1} (v+x+n)^{l-s} \cdot x^{k-l} \cdot(k-l+1)_l}{(1-s)_l} \, +  \, \notag\\[2mm]
&&\displaystyle\qquad
+ \,\frac{(-1)^{k+1} (v+x+n)^{k-s} \cdot (k-1)! }{(1-s)_k} \,,
\end{eqnarray}
and evaluating the infinite series $\sum F_{n,1,0}[k\,x^{k-1}]$,
formula \eqref{j09384jhfd} reduces to the desired result stated in \eqref{984rhc44}, since $\,(1-s)_l=(-1)^l (s-l)_l$.
Note that we have \eqref{jxh293nbwes} as a particular case of \eqref{984rhc44} at $k=1$. Moreover,
if we put $v=1$ and simplify the second sum in the first line, then we immediately arrive 
at this curious formula for the $\zeta$-function
\be\label{i0u023jddf}
\zeta(s) \,=\,-\!\sum_{l=1}^{k-1} \frac{\,(k-l+1)_l\,}{(s-l)_l} \cdot\zeta(s-l)  \,+ 
\,\frac{k}{s-k} \,+\,k\! \sum_{n=0}^\infty G_{n+1}^{(k)}\Delta^n 1^{-s}\,,\qquad
\ee
which is also announced in Theorem \ref{45yvgt4rf45}. 

It is interesting to remark that formula (48) from \cite[p.~414]{iaroslav_08} also contains similar shifted values of the $\zeta$-functions.\footnote{Formula 
 (48) from \cite{iaroslav_08} and its proof were first released on 5 January 2015 in the 6th arXiv version of the paper. 
28 September 2015, a particular case of the same formula for nonnegative integer $s$ 
was also presented by Xu, Yan and Shi \cite[p.~94, Theorem 2.9]{xu_01}, who, apparently,
were not aware of the arXiv preprint of our work \cite{iaroslav_08} (we have not found the preprint of \cite{xu_01}).} 
Another functional relationship of the same kind was discovered by Allouche et al.~in \cite[Sect.~3, p.~362]{allouche_03}.
\end{proof}

\noindent{\bfseries Nota Bene.}\label{hg8fgvikhvgc6}
The numbers $G_n^{(k)}$ are yet another generalization of Gregory's coefficients $G_n=G_n^{(1)}$,
and we already encountered them in  \cite[pp.~413--414]{iaroslav_08}.
In the latter, we, \emph{inter alia}, showed that $\,\big|G_n^{(k)}\big|\sim n^{-1}\ln^{-k-1}n\,$ at $\,n\to\infty\,$ 
and also evaluated the series $\sum (-1)^n G_n^{(k)}\!/n \,$ for $k=2, 3, 4,\ldots$ \cite[Eqs.~38, 50]{iaroslav_08}.  Besides,
although it is much easier to define $G_n^{(k)}$ via the finte sum with the 
Stirling numbers of the first kind \eqref{utrdswf9}, they equally may be defined via the generating function
\be\label{90u2hl3f}
z\,\frac{\,(-1)^{k+1}(k-1)! \,}{\,\ln^k(1+z)\,} \,
+\, (1+z)\!\sum_{l=1}^{k-1} \frac{(-1)^{l+1}(k-l+1)_{l-1}\,}{\,\ln^l(1+z)\,} 
\,=\,\frac{1}{\,k\,} \,+ \sum_{n=1}^\infty G_n^{(k)} z^n \,,\quad
\ee
where $|z|<1\,$ and $\,k\in\mathbbm{N}\,$.\footnote{For $k=1$
the sum over $l$ should be taken as $0$.}~This
formula may be obtained by proceeding in the manner analogous to \eqref{8u4rfhq2me}
and coincides with \eqref{eq32} at $k=1.$ 
Note that the left part of \eqref{90u2hl3f} 
contains the powers of $\ln^{-1}(1+z)$ up to and including $k$,
the fact which prompted us to call $G_n^{(k)}$ \emph{Gregory's coefficients of higher order}.

\begin{theorem}\label{ujihd97823hdbe}
Let $\,\boldsymbol{u}\equiv\{u_n\}_{n\geqslant0}\,$ be the sequence of 
bounded complex numbers
and let denote \mbox{$\,\Delta u_n=u_{n+1}-u_n\,$}, $\,\Delta^2 u_n=\Delta(\Delta u_n) = u_{n+2} - 2u_{n+1} + u_n\,,$ 
and so on, exactly as in \eqref{0i3u40jfmnr}.
Consider now the following functional 
\be\notag
\zeta(s,v,\boldsymbol{u})=\sum_{n=0}^\infty \frac{u_n}{\, (v+n)^{s}}\,, \qquad 
\begin{array}{l}
\Re s>1\, \\[1mm]
v\in\mathbbm{C}\setminus\!\{0,-1,-2,\ldots\}\,,
\end{array}
\ee
generalizing the Euler--Riemann $\zeta$-function $\zeta(s)=\zeta(s,1,\boldsymbol{1})$, 
the Hurwitz $\zeta$-function $\zeta(s,v)=\zeta(s,v,\boldsymbol{1})$, 
the functional $F(s)$ from \cite[Sect.~3, p.~362]{allouche_03} $F(s)=\zeta(s,1,\boldsymbol{u})$.
The Dirichlet series $\,\zeta(s,v,\boldsymbol{u})\,$ admits two following 
series representations
\begin{eqnarray}
&&\displaystyle \zeta(s,v,\boldsymbol{u})\,=
\frac{1}{\,m\,(s-1)\,}\sum_{n=0}^{m-1} \zeta(s-1,v+a+n+1,\Delta\boldsymbol{u}) \,+
\qquad\notag \\[1mm]
\displaystyle&&\displaystyle\qquad
+ \, \frac{u_0}{\,m\,(s-1)\,} \! \sum_{n=0}^{m-1}  (v+a+n)^{1-s}
- \, \frac{1}{\,m\,}\!\sum_{n=1}^\infty (-1)^n N_{n,m}(a)\, \zeta(s,v+n, \Delta^n \boldsymbol{u})\, +\notag\\[1mm]
\displaystyle&&\displaystyle\qquad
+ \,\frac{1}{\,m\,}\!\sum_{n=0}^\infty  (-1)^{n}N_{n+1,m}(a)\sum_{l=0}^{n} \Delta^l u_0 
\sum_{k=0}^{n-l} (-1)^k\binom{n-l}{k} (k+l+v)^{-s}  \label{uygf8fgffcdy6} \\[1mm]\notag
\end{eqnarray}
and 
\begin{eqnarray}
\zeta(s,v,\boldsymbol{u})&&\displaystyle \! =\,
\frac{1}{\,m\,(s-1)\,}\!\sum_{n=0}^{m-1}\!\Big\{ \zeta(s-1,v+a+n+1,\Delta\boldsymbol{u}) \,+\, u_0(v+a+n)^{1-s}\Big\} + \notag\\[1mm]
\displaystyle&&\displaystyle\qquad
+\, \frac{1}{m}\!\sum_{n=1}^\infty N_{n,m}(a)\, \Delta^{n-1}\zeta(s,v+1, \Delta \boldsymbol{u}) \,+ \notag\\[1mm]
\displaystyle&&\displaystyle\qquad 
+ \,\frac{u_0}{m}\!\sum_{n=0}^\infty (-1)^n N_{n+1,m}(a) \!\sum_{k=0}^{n} (-1)^k \binom{n}{k} (k+v)^{-s}   \,,\qquad\label{ft76rfvcvc}
\end{eqnarray}
where $\,\Delta^n\boldsymbol{u}\equiv\{\Delta^n u_l\}_{l\geqslant0}\,$ and the operator $\,\Delta^0\,$ does nothing by convention.
\end{theorem}

\begin{proof}
From the definition of $\zeta(s,v,\boldsymbol{u})$ it follows 
that this Dirichlet series possesses the following recurrence relation in $v$
\begin{eqnarray}
\label{98ryh34f23}
\Delta \zeta(s,v,\boldsymbol{u}) \,\equiv
\,\zeta(s,v+1,\boldsymbol{u})-\zeta(s,v,\boldsymbol{u}) \! && = 
\sum_{n=0}^\infty \frac{-\Delta u_n}{\, (v+n+1)^{s}} - u_0v^{-s} \notag\\[3mm]
&& 
=-\zeta(s,v+1,\boldsymbol{\Delta u})- u_0 v^{-s}\,.\qquad\qquad
\end{eqnarray}
The second--order recurrence relation reads
\be\notag
\Delta^2 \zeta(s,v,\boldsymbol{u}) \,\equiv\,\Delta \big(\Delta \zeta(s,v,\boldsymbol{u})\big)
=\,\zeta(s,v+2, \Delta^2 \boldsymbol{u}) + \Delta u_0 (v+1)^{-s} - u_0 \Delta v^{-s}\,,
\ee
and more generally we have
\be\notag
\Delta^k \zeta(s,v,\boldsymbol{u}) =\,(-1)^k\zeta(s,v+k, \Delta^k \boldsymbol{u}) 
+ \sum_{l=0}^{k-1}  (-1)^{l+1}\!\cdot \Delta^l u_0 \cdot \Delta^{k-1-l} (v+l)^{-s} \,.
\ee
Writing the Gregory--Newton interpolation formula for $\,\zeta(s,x+v,\boldsymbol{u})\,$ we, therefore, obtain
\begin{eqnarray}
& \zeta(s,x+v,\boldsymbol{u}) & =\label{2uy39dghdbg}
\,\zeta(s,v,\boldsymbol{u})+ \sum_{n=1}^\infty \frac{(x-n+1)_n}{n!} \,\Delta^n \zeta(s,v,\boldsymbol{u}) \,= \notag\\[1mm]
\displaystyle
&&=
\,\zeta(s,v,\boldsymbol{u})+\sum_{n=1}^\infty \frac{(-1)^n(x-n+1)_n}{n!} \,\zeta(s,v+n, \Delta^n \boldsymbol{u}) + \notag\\[1mm]
\displaystyle&&\quad
+\sum_{n=1}^\infty \frac{(x-n+1)_n}{n!}  
\sum_{l=0}^{n-1}  (-1)^{l+1}\!\cdot \Delta^l u_0 \cdot \Delta^{n-1-l} (v+l)^{-s} .\notag\qquad\qquad
\end{eqnarray}
Effecting the term--by--term integration over the interval $x\in[a,a+m]$ yields
\begin{eqnarray}
&&\displaystyle \zeta(s,v,\boldsymbol{u})\,=
\frac{1}{\,m\,(s-1)\,}\Big\{\zeta(s-1,v+a,\boldsymbol{u}) - \zeta(s-1,v+a+m,\boldsymbol{u}) \Big\} - \notag\\[1mm]
\displaystyle&&\displaystyle\qquad\qquad\qquad\qquad
- \frac{1}{m}\!\sum_{n=1}^\infty (-1)^n N_{n,m}(a)\, \zeta(s,v+n, \Delta^n \boldsymbol{u}) \, - \notag
\end{eqnarray}
\begin{eqnarray}
\displaystyle&&\displaystyle\qquad\qquad\qquad\qquad
- \frac{1}{m}\!\sum_{n=1}^\infty N_{n,m}(a)\sum_{l=0}^{n-1} 
(-1)^{l+1}\!\cdot \Delta^l u_0 \cdot \Delta^{n-1-l} (v+l)^{-s}  .
\qquad\qquad \notag
\end{eqnarray}
Finally, simplifying the expression in curly brackets 
\begin{eqnarray}
&&\displaystyle 
\zeta(s-1,v+a,\boldsymbol{u}) - \zeta(s-1,v+a+m,\boldsymbol{u}) \,=\label{8734rgb3e487y}\\[2mm]
&&\displaystyle \qquad\qquad
= \sum_{n=0}^{m-1} \underbrace{\Big[\zeta(s-1,v+a+n,\boldsymbol{u}) - \zeta(s-1,v+a+n+1,\boldsymbol{u}) \Big] 
}_{-\Delta\zeta(s-1,v+a+n,\boldsymbol{u}) }= \qquad\qquad\notag \\[1mm]
&&\displaystyle \qquad\qquad
\,=\sum_{n=0}^{m-1} \zeta(s-1,v+a+n+1,\Delta\boldsymbol{u}) \,
+ u_0\!\sum_{n=0}^{m-1} \! (v+a+n)^{1-s} , \qquad\qquad\notag 
\end{eqnarray}
we arrive at 
\begin{eqnarray}
&&\displaystyle \zeta(s,v,\boldsymbol{u})\,=
\frac{1}{\,m\,(s-1)\,}\!\sum_{n=0}^{m-1} \!\Big\{\zeta(s-1,v+a+n+1,\Delta\boldsymbol{u}) \,+
u_0 (v+a+n)^{1-s}\Big\} - \notag\\[1mm]
\displaystyle&&\displaystyle\qquad\qquad\qquad 
- \frac{1}{m}\!\sum_{n=1}^\infty (-1)^n N_{n,m}(a)\, \zeta(s,v+n, \Delta^n \boldsymbol{u})\,- \label{ghiht57yf}\\[1mm]
\displaystyle&&\displaystyle\qquad\qquad\qquad
- \frac{1}{m}\!\sum_{n=1}^\infty N_{n,m}(a)\sum_{l=0}^{n-1}  (-1)^{l+1}\!\cdot \Delta^l u_0 
\cdot \Delta^{n-1-l} (v+l)^{-s}  ,
\notag
\end{eqnarray}
which is a generalization of \eqref{0934uf023jhd} to the Dirichlet series $\,\zeta(s,v,\boldsymbol{u})\,$.\qed

Now, from Eq.~\eqref{2uy39dghdbg} we may also proceed in a slightly different manner.
Applying operator $\Delta^{n-1}$ to \eqref{98ryh34f23} gives
\be\notag
\Delta^n \zeta(s,v,\boldsymbol{u}) 
=-\Delta^{n-1}\zeta(s,v+1,\boldsymbol{\Delta u})- u_0 \Delta^{n-1}v^{-s}.
\ee
Therefore, \eqref{2uy39dghdbg} becomes
\begin{eqnarray}
&&\displaystyle
\zeta(s,x+v,\boldsymbol{u}) \,=
\,\zeta(s,v,\boldsymbol{u}) - \sum_{n=1}^\infty \frac{(x-n+1)_n}{n!} \,\Delta^{n-1}\zeta(s,v+1,\boldsymbol{\Delta u}) \,- \notag\\[2mm]
\displaystyle&&\displaystyle\qquad\qquad\qquad\quad
- u_0\!\sum_{n=1}^\infty \frac{(x-n+1)_n}{n!} \,\Delta^{n-1}v^{-s}.
\end{eqnarray}
Integrating termwise over $x\in[a,a+m]$ and accounting for \eqref{8734rgb3e487y} yields
\begin{eqnarray}
&&\displaystyle \zeta(s,v,\boldsymbol{u})\,=\,
\frac{1}{\,m\,(s-1)\,}\!\sum_{n=0}^{m-1}\!\Big\{ \zeta(s-1,v+a+n+1,\Delta\boldsymbol{u}) \,+ u_0(v+a+n)^{1-s}\Big\} + \notag\\[1mm]
\displaystyle&&\displaystyle\qquad
+ \frac{1}{m}\!\sum_{n=1}^\infty N_{n,m}(a)\, \Delta^{n-1}\zeta(s,v+1, \Delta \boldsymbol{u})
+ \frac{u_0}{m}\!\sum_{n=1}^\infty N_{n,m}(a) \,\Delta^{n-1} v^{-s}   \,,\label{73dyd2odhedy}
\end{eqnarray}
which, unlike \eqref{ghiht57yf}, 
directly relates $\,\zeta(\ldots,\ldots,\boldsymbol{u})\,$ to $\,\zeta(\ldots,\ldots,\Delta\boldsymbol{u}).$
\end{proof}

Although formul\ae~\eqref{ghiht57yf}, \eqref{73dyd2odhedy} may seem quite theoretical, they both have a multitude of interesting applications
and consequences for the concrete $L$-functions. For example, if $u_n$ is a polynomial of degree $k-1$ in $n$,
the $k$-th finite difference $\,\Delta^k\boldsymbol{u}\,$ vanishes and so do higher 
differences. Hence $\,\zeta(s,v+n, \Delta^n \boldsymbol{u})=0\,$ for $n\geqslant k$
and the right part of the previous equation becomes much simpler. 
On the other hand, if $u_n$ is a polynomial in $n$, the considered $L$-function may always be written as a linear combination 
of the $\zeta$-functions, so that we have a relation between the $\zeta$-functions of arguments $s, s-1, s-2, \ldots\,$
We come to discuss some of such examples in the next Corollary.

\begin{corollary}\label{908wq2jhds}
The following formul\ae~relating the $\zeta$-functions of different arguments and types hold
\begin{eqnarray}
&&\displaystyle \big(v+a+\tfrac{1}{2}m-1\big)\!\cdot\zeta(s,v) \,=\,-\frac{\zeta(s-1,v+a)}{\,s-1\,}\,
+ \,\zeta(s-1,v)\,+ \qquad\qquad\label{jic340jnd4}\\[1mm]
&&\displaystyle \qquad\qquad\qquad
+ \,\frac{1}{\,m\,(s-1)\,}\!\sum_{n=0}^{m-1} \frac{\,m-n-1\,}{\,(v+a+n)^{s-1}} \,+ \notag\\[1mm]
&&\displaystyle \qquad\qquad\qquad\notag
+ \,\frac{1}{m}\!\sum_{n=0}^\infty (-1)^n N_{n+2,m}(a) 
\!\sum_{k=0}^{n} (-1)^k \binom{n}{k} (k+v)^{-s}   \,,
\end{eqnarray}

\begin{eqnarray}
&&\displaystyle
\big(v+a-\tfrac{1}{2}\big)\!\cdot\zeta(s,v) \,=\,
-\frac{\,\zeta(s-1,v+a)\,}{\,s-1\,} \,+ \,\zeta(s-1,v)\, +   \label{f234i94m4rf}\\[2mm]
&&\displaystyle\qquad\qquad\qquad\qquad\qquad
+ \sum_{n=0}^\infty (-1)^n \psi_{n+2}(a) \!  \sum_{k=0}^{n} (-1)^k \binom{n}{k} (k+v)^{-s}   \,, \notag
\end{eqnarray}

\begin{eqnarray}
&&\displaystyle \label{9u8y943dh34hb2}
\big(v+a\big)\!\cdot\zeta(s,v) =\frac{(v+a)^{1-s}}{\,2\,(s-1)\,}
-\frac{\,\zeta(s-1,v+a)\,}{\,s-1\,}  \,+ \,\zeta(s-1,v)\, +\\[2mm]
&&\displaystyle\qquad\qquad\qquad\qquad\qquad
+ \frac{1}{2}\!\sum_{n=0}^\infty (-1)^n N_{n+2,2}(a) 
\! \sum_{k=0}^{n} (-1)^k \binom{n}{k} (k+v)^{-s} \,,\notag
\end{eqnarray}

\be\label{9u8y943dh34hb3}
\big(v-\tfrac{1}{2}\big)\!\cdot\zeta(s,v) =
\frac{s-2}{\,s-1\,}\cdot\zeta(s-1,v)
+ \sum_{n=0}^\infty (-1)^n G_{n+2}  \!  \sum_{k=0}^{n} (-1)^k \binom{n}{k} (k+v)^{-s}    \,,\qquad
\ee

\begin{eqnarray}
\frac{m}{2}\cdot\zeta(s) \,=\,\frac{s-2}{\,s-1\,}\cdot\zeta(s-1)\!\!
&&\displaystyle + \,\frac{\,m H_m^{(s-1)} - H_m^{(s-2)}\,}{\,m\,(s-1)\,} \, + \label{jhf39u4hfnr}
\end{eqnarray}
\begin{eqnarray}
&&\displaystyle 
+ \,\frac{1}{m}\!\sum_{n=0}^\infty (-1)^n N_{n+2,m}(0) \!
\sum_{k=0}^{n} (-1)^k \binom{n}{k} (k+1)^{-s} ,\notag
\end{eqnarray}

\be\label{po2c43i0dj4}
\zeta(s) \,=\,\frac{2(s-2)}{\,s-1\,}\cdot\zeta(s-1)\,
+ \,2\!\sum_{n=0}^\infty (-1)^n  G_{n+2} \!\sum_{k=0}^{n} (-1)^k \binom{n}{k} (k+1)^{-s} ,
\ee

\begin{eqnarray}
\frac{\zeta(s-1,a)}{\,s-1\,} &&\displaystyle\!\!
=\,\frac{a^{1-s}}{\,s-1\,}\,-\big(a+\tfrac{1}{2}m\big)\!\cdot\zeta(s)
+ \,\zeta(s-1)\,+ \label{983ygf39}\notag\\[2mm]
&&\displaystyle \qquad\qquad
+ \,\frac{1}{\,m\,(s-1)\,}\!\sum_{n=1}^{m} \frac{\,m-n\,}{\,(a+n)^{s-1}} \, +\\[2mm]
&&\displaystyle \qquad\qquad
+ \,\frac{1}{m}\!\sum_{n=0}^\infty  (-1)^n N_{n+2,m}(a) \!\sum_{k=0}^{n} (-1)^k \binom{n}{k} (k+1)^{-s}  \,,\notag
\end{eqnarray}
where $m$ is a natural number and $H_m^{(s)}$ is the generalized harmonic number
(see p.~\pageref{ufh3984fbnf}).
\end{corollary}

\begin{proof}
Let $u_n=P_1(n)=\alpha + \beta n$, where $\alpha$ and $\beta$ are some coefficients.
Then,
\be\notag
u_0=\alpha\,,\qquad \Delta u_n=\beta ,\qquad \Delta^2 u_n=0\,,\ldots, \quad \Delta^k u_n=0 \,,
\ee
and hence $\,\zeta(s,v, \Delta \boldsymbol{u})=\beta\zeta(s,v)\,$. On the other hand, 
from a simple arithmetic argument it also follows that
\be\notag
\zeta(s,v,\boldsymbol{u})\,=\,(\alpha - \beta v) \,\zeta(s,v) \,+\, \beta \,\zeta(s-1,v)\,.
\ee
Using the recurrence relation of the Hurwitz $\zeta$-function and the fact that
\be
\Delta^{n-1} \zeta(s,v+1) \,=\,-\Delta^{n-2}v^{-s} - \Delta^{n-1}v^{-s} ,
\ee
as well as recalling that
$N_{1,m}(a) = ma+\frac{1}{2}m^2$, see \eqref{uyh293dd}, and setting for simplicity $\alpha=1$
and $\beta=1$,\footnote{It is possible to perform these calculations with arbitrary $\alpha$
and $\beta$, but the resulting expressions become very cumbersome.} formula \eqref{73dyd2odhedy} becomes
\begin{eqnarray}
&&\displaystyle \big(1-v-a-\tfrac{1}{2}m\big)\zeta(s,v) \, =\,
\frac{1}{\,m\,(s-1)\,}\!\sum_{n=0}^{m-1}\!\Big\{  \zeta(s-1,v+a+n+1)  \, + \Big.    \notag\\[1.5mm]
&&\displaystyle\qquad 
\Big. +\,  (v+a+n)^{1-s}\Big\} 
\,- \,\zeta(s-1,v)\,
- \,\frac{1}{m}\!\sum_{n=0}^\infty N_{n+2,m}(a) \,\Delta^{n} v^{-s}  
\,. \qquad\label{73dyd2}\\[0.8mm]\notag
\end{eqnarray}
By recursively decreasing the second argument in $\,\zeta(s-1,v+a+n+1) \,$,
we obtain the following functional relationship
\begin{eqnarray}
&&\displaystyle \big(1-v-a-\tfrac{1}{2}m\big)\zeta(s,v) \,=\,\frac{\zeta(s-1,v+a)}{\,s-1\,}\,
- \,\zeta(s-1,v)\,- \label{983ygf39b}\\[0.8mm]
&&\displaystyle \qquad\qquad
- \,\frac{1}{\,m\,(s-1)\,}\!\sum_{n=0}^{m-1} \frac{\,m-n-1\,}{\,(v+a+n)^{s-1}} 
- \,\frac{1}{m}\!\sum_{n=0}^\infty N_{n+2,m}(a) \,\Delta^{n} v^{-s}   \,,\notag
\end{eqnarray}
where $m\in\mathbbm{N}$ and the first sum in the second line should be taken as nothing
for $m=1$. This is our formula \eqref{jic340jnd4} and it has many interesting particular cases.
For instance, for $m=1$ and $m=2$ we obtain \eqref{f234i94m4rf} and \eqref{9u8y943dh34hb2}
respectively, since $N_{n+2,1}(a) =\psi_{n+2}(a)$. Furthermore, making $a=0$ we get \eqref{9u8y943dh34hb3},
because $\psi_{n}(0)=G_n.$ Setting $v=1$ and $a\in\mathbbm{N}_0$ 
gives us corresponding expressions for $\zeta(s)$. For example, at $v=1$ and $a=0$ 
relationship \eqref{983ygf39b} becomes
\be
\frac{m}{2}\zeta(s) =\frac{s-2}{\,s-1\,}\zeta(s-1)
+\frac{1}{\,m\,(s-1)\,}\!\sum_{n=1}^{m} \!\frac{\,m-n\,}{\,n^{s-1}} 
+ \frac{1}{m}\!\sum_{n=0}^\infty N_{n+2,m}(0) \,\Delta^{n} 1^{-s} , \quad
\ee
where the finite sum in the middle may be also written in terms of the generalized harmonic numbers;
we, thus, arrive at \eqref{jhf39u4hfnr}.
Furthermore, putting $m=1$ we obtain a strikingly simple funtional relationship 
with Gregory's coefficients, equation \eqref{po2c43i0dj4}.
From \eqref{983ygf39b} we may also obtain a relationship between $\zeta(s)$ and $\zeta(s,a)$. Putting $v=1$
gives us relationship \eqref{983ygf39}. Finally, proceeding similarly with $u_n=P_2(n)$ and using 
\be\notag
\Delta^n \zeta(s,v,\boldsymbol{u}) \,
=\,\Delta^{n-2}\zeta(s,v+2, \Delta^2 \boldsymbol{u}) + \Delta u_0 \Delta^{n-2} (v+1)^{-s} - u_0 \Delta^{n-1} v^{-s}\,,
\ee
we may obtain formul\ae~relating $\zeta$-functions 
of arguments $s, s-1$ and $s-2$. The same procedure may be applied to $u_n=P_3(n)$ and so on.
\end{proof}

\begin{corollary}
The generalized Stieltjes constants $\gamma_m(v)\,$ may be given by the following series representations
\begin{eqnarray}\label{kdx2je}
\gamma_m(v)\, =\,-\frac{\ln^{m+1} v}{\,m+1\,}+\sum_{n=0}^\infty\big| G_{n+1}\big
|\sum_{k=0}^{n} (-1)^k \binom{n}{k}\frac{\ln^m (k+v)}{k+v} \,,
\end{eqnarray}
where $\Re v >0$,
\begin{eqnarray}\label{fds5tfde4}
\gamma_m(v)\,=\,-\frac{\ln^{m+1} (v-1)}{\,m+1\,} -\sum_{n=0}^\infty C_{n+1}
\sum_{k=0}^{n} (-1)^k \binom{n}{k}\frac{\ln^m (k+v)}{k+v}\,,
\end{eqnarray}
where $\Re v >1$, 
\begin{eqnarray}\label{jc309u4n2ol3}
\gamma_m(v)\!&&=\,-\frac{1}{r\,(m+1)}\sum_{l=0}^{r-1}\ln^{m+1}(v+a+l) \, + \\[0.8mm]
&&\qquad\quad + \,\frac{1}{\,r\,}\!
\sum_{n=0}^\infty (-1)^n N_{n+1,r}(a)
\sum_{k=0}^{n} (-1)^k \binom{n}{k}\frac{\ln^m (k+v)}{k+v}\,, \notag
\end{eqnarray}
where $r\in\mathbbm{N}$, $\Re a > -1$ and $\Re v >-\Re a$, and 
\begin{eqnarray}
&&\displaystyle \gamma_m(v)\,=\,\frac{1}{\,1-v-a-\tfrac{1}{2}r\,}
\left\{\frac{(-1)^m}{\,m+1\,}\,\zeta^{(m+1)}(0,v+a)\,-\, (-1)^m \zeta^{(m)}(0,v)\, 
+  \vphantom{\sum_a^b}   \right. \qquad\notag\\[2mm]
&&\displaystyle\qquad\qquad\qquad\qquad
+ \,\frac{1}{\,r\,(m+1)\,}\!\sum_{n=0}^{r-2} (r-n-1)\ln^{m+1}(v+a+n) \, -  \notag\\[2mm]
&&\displaystyle\qquad\qquad \qquad\qquad
\left.  - \,\frac{1}{\, r\,}\!\sum_{n=0}^\infty (-1)^n N_{n+2,r}(a) 
\sum_{k=0}^{n} (-1)^k \binom{n}{k}\frac{\ln^m (k+v)}{k+v}\right\} \label{ilu42798hd2ibd}
\end{eqnarray}
under the same conditions.
\end{corollary}

\begin{proof}
The generalized Stieltjes constants $\gamma_m(v)\,$, $m\in\mathbbm{N}_0$, 
$v\in\mathbbm{C}\setminus\!\{0,-1,-2,\ldots\}$,
are introduced analogously to the ordinary Stieltjes constants
\begin{eqnarray}
\label{dhd73vj6s1b}
\zeta(s,v)\,=\,\frac{1}{\,s-1\,} -\Psi(v) + \sum_{m=1}^\infty\frac{(-1)^m\gamma_m(v)
}{m!}\,(s-1)^m \,, \qquad s\in\mathbbm{C}\setminus\!\{1\},\quad
\end{eqnarray}
with  $\gamma_0(v)=-\Psi(v)\,$, see e.g.~\cite[p.~541, Eq.~(14)]{iaroslav_07}.\footnote{For more
information on $\gamma_m(v)\,$, see \cite{iaroslav_06}, \cite{iaroslav_07}, and the literature given in 
the last reference. Note that since $\zeta(s,1)=\zeta(s)$, 
the generalized Stieltjes constants $\gamma_m(1)=\gamma_m$.}~Thus, 
from expansions \eqref{jxh293nbwes}, \eqref{jx20nsw}, \eqref{98y6b796} and \eqref{983ygf39b}, 
by proceeding in the same manner as in Corollary \ref{ihy923hbs}, 
we deduce the announced series representations.\footnote{It may also
be noted that the particular case $m=1$ of \eqref{kdx2je} 
was earlier given by Coffey \cite[p.~2052,  Eq.~(1.18)]{coffey_08}.}~Notice also 
that the above formul\ae~may be rewritten in a slightly different way by means of the recurrence relation 
for the generalized Stieltjes constants $\,\gamma_m(v+1)=\gamma_m(v)-v^{-1}\ln^m \!v\,$.
\end{proof}

\vspace{1em}
\noindent{\bfseries Remark 2, related to the digamma function (the $\Psi$-function) and 
to Euler's constant ${\boldsymbol\gamma}$.}
Since 
\be
(-1)^n\Delta^n v^{-1}\,=\,\frac{n!}{(v)_{n+1}}
\ee
we have for the zeroth Stieltjes constant, 
and hence for the digamma function $\Psi(v)$, the following expansions
\begin{eqnarray}\label{jnf394nf3}
&& \Psi(v)\, =\,\ln v- \sum_{n=1}^\infty\frac{\big| G_{n}\big|\,(n-1)!}{(v)_{n}}\,,\qquad 
\Re v >0\,, \\[1.5mm]
&& \label{d2398he}
\Psi(v)\,=\,\ln(v-1) +
\sum_{n=1}^\infty\frac{C_{n}\,(n-1)!}{(v)_{n}}\,,\qquad 
\Re v >1\,, \\[1.5mm]
&& 
\Psi(v)\,=\,\ln(v+a) + 
\sum_{n=1}^\infty\frac{(-1)^n\,\psi_{n}(a)\,(n-1)!}{(v)_{n}}\,,\qquad \label{890u4h34f0}  
\end{eqnarray}
\begin{eqnarray}
&& 
\Psi(v)\,=\, \frac{1}{\,r\,}\!\sum_{l=0}^{r-1}\ln(v+a+l) + \frac{1}{\,r\,}\!
\sum_{n=1}^\infty\frac{(-1)^n\,N_{n,r}(a)\,(n-1)!}{(v)_{n}}\,,\qquad \label{890u4h34f}\\[2mm]
&&
\Psi(v)\,=\, \frac{1}{\,\tfrac{1}{2}r+v+a-1\,}
\left\{\ln\Gamma(v+a) + v - \frac12\ln2\pi - \frac12 +  \vphantom{\sum_a^b}\right.\notag\\[2mm]
&&\displaystyle \qquad
\left. 
+\, \frac{1}{\,r\,}\!\sum_{n=0}^{r-2} (r-n-1)\ln(v+a+n) \, 
+\, \frac{1}{\,r\,}\!\sum_{n=1}^\infty\frac{(-1)^n\,N_{n+1,r}(a)\,(n-1)!}{(v)_{n}}\right\}\,,\qquad
\label{kih2g3o8dgg2ikhbde}
\end{eqnarray}
with $r\in\mathbbm{N}$,  $\Re a>-1$ and $\Re v >-a$, because of \eqref{897tv87tc}.\footnote{In the last formula the sum in the middle should be taken as zero for $r=1$.
Note also that $\,\zeta(0,x)=\frac12-x\,$ and $\,\zeta'(0,x)=\ln\Gamma(x)-\frac12\ln2\pi\,$. 
Similarly to \eqref{890u4h34f0} formul\ae~\eqref{jc309u4n2ol3}--\eqref{ilu42798hd2ibd} and \eqref{kih2g3o8dgg2ikhbde} may also be written
in terms of $\psi_{n}(a)$ instead of $N_{n,r}(a)$ when $r=1$, see \eqref{897tv87tc}.}
First two representations coincide with the not well-known Binet--N{\o}rlund expansions for the digamma function 
\cite[pp.~428--429, Eqs.~(91)--(94)]{iaroslav_08},\footnote{Formula \eqref{d2398he}
reduces to \cite[p.~429, Eq.~(94), first formula]{iaroslav_08} by putting $v$ instead of $v-1$ and by making
use of the recurrence relationship for the digamma function $\Psi(v+1)=\Psi(v)+v^{-1}$.}
while \eqref{890u4h34f0}--\eqref{kih2g3o8dgg2ikhbde} seem to be new. From \eqref{jnf394nf3} and \eqref{d2398he}, it also immediately follows that
$\,\ln(v-1) < \Psi(v) < \ln v\,$ for $\, v>1\,$,
since the sums with $G_n$ and $C_n$ keep their sign.\footnote{This simple and important result is not new, but 
its derivation from \eqref{jnf394nf3} and \eqref{d2398he} seems to be novel, and in addition, is elementary.}
We may also obtain series expansions for the digamma function from \eqref{984rhc44}, but the resulting expressions strongly depend on $k$.
For instance, putting $k=2$ and expanding both sides into the Laurent series \eqref{dhd73vj6s1b}, we obtain the following formula
\be\label{h9843dh2edf}
\Psi(v)\, =\,2\ln\Gamma(v) - 2v\ln v + 2v +2\ln v -\ln2\pi + 2\!\sum_{n=1}^\infty\frac{ (-1)^n \, G_{n}^{(2)}\,(n-1)!}{(v)_{n}}\,, 
\ee
which relates the $\Gamma$-function to its logarithmic derivative.\footnote{There were many attempts
aiming to find possible relationships between these two functions. For instance, in 1842 Carl Malmsten, by trying to find such 
a relationship, obtained a variant of Gauss' theorem for $\Psi(v)$ at $v\in\mathbbm{Q}$, see \cite[p.~37, Eq.~(23)]{iaroslav_06},
\cite[p.~584, Eq.~(B.4)]{iaroslav_07}.}
For higher $k$ these expressions become quite cumbersome and also imply the derivatives of the $\zeta$-function at negative integers.
In particular, for $k=3$, we deduce
\begin{eqnarray}\label{h9823eh3r}
&&\Psi(v)\, =\,3\ln\Gamma(v) - 6\zeta'(-1,v) + 3v^2\ln{v} -  \frac32 v^2 - 6v\ln(v)+  \\[2mm]
&&\qquad\qquad \displaystyle +3 v+3\ln{v} - \frac32\ln2\pi + \frac12  
+ 3\!\sum_{n=1}^\infty\frac{ (-1)^n \, G_{n}^{(3)}\,(n-1)!}{(v)_{n}}\,,\notag
\end{eqnarray}
provided the convergence of the last series.
Formula \eqref{h9843dh2edf} is also interesting in that it gives series with rational terms for $\ln\Gamma(v)$
if $v\in\mathbbm{Q}$ (we only need to use Gauss' digamma theorem for this \cite[p.~584, Eq.~(B.4)]{iaroslav_07}). 
Note also that all series \eqref{jnf394nf3}--\eqref{h9823eh3r} converge very rapidly for large $v$. For instance, 
putting in \eqref{890u4h34f} or in \eqref{kih2g3o8dgg2ikhbde} $\,v=2\pi\,$, $a=3$, $m=2$, and taking only 10 terms in the last sum, 
we get the value of $\Psi(2\pi)$ with 9 correct digits.

Putting in the previous formul\ae~for the digamma function argument $v\in\mathbbm{Q}$, we may also obtain series with rational terms 
for Euler's constant. The most simple is to put $v=1$. In this case, formula \eqref{jnf394nf3} reduces
to the famous Fontana--Mascheroni series
\be\label{09384rtjh3}
\Psi(1)\,=\,-\gamma\,=\,-\sum_{n=1}^\infty \frac{\big| G_n\big|}{n}\,,
\ee
see e.g.~\cite[p.~207]{vacca_02}, \cite[p.~539]{iaroslav_07}, 
\cite[pp.~406, 413,429--430]{iaroslav_08}, \cite[p.~379]{iaroslav_09}, while \eqref{890u4h34f} gives us
\be\label{8769f657fde6}
\gamma\,=\, -\frac{1}{m}\sum_{l=1}^{m}\ln(a+l) -
\frac{1}{m}\sum_{n=1}^\infty\frac{\,(-1)^n\,N_{n,m}(a)\,}{\,n\,}\,,\qquad 
\begin{array}{l}
m\in\mathbbm{N}\\
a>-1 
\end{array}\,.
\ee
This series generalizes \eqref{09384rtjh3} to a large family of series (we have the aforementioned 
series at $a=0$ since $N_{n,1}(0)=\psi_n(0)=G_n$).
For example, setting $a=-\frac12$ and $m=1$ (the mean value between $a=-1$ corresponding to the coefficients $C_n$ and $a=0$ 
corresponding to $G_n$), we have, by virtue of \eqref{897tv87tc}, the following series 
\begin{eqnarray}
&\gamma\,&\displaystyle
=\, \ln2 - \sum_{n=1}^\infty\frac{\,(-1)^n\,N_{n,1}\big(-\frac12\big)\,}{\,n\,}\,
=\,\ln2 - \sum_{n=1}^\infty\frac{\,(-1)^n\,\psi_{n}\big(-\frac12\big)\,}{\,n\,}\,= \notag\\[1mm]
&&\displaystyle
=\ln2 -  0 - \frac{1}{48} - \frac{1}{72} - \frac{223}{23\,040} - \frac{103}{14\,400} - \frac{32\,119}{5\,806\,080}
- \frac{1111}{250\,880} - \ldots\qquad
\end{eqnarray}
relating two fundamental constants $\gamma$ and $\ln2$. This series, however, converges quite slowly, as $\,\sum n^{-\nicefrac{3}{2}}\ln^{-1}\!n\,$
by virtue of \eqref{809q0g08g}. A more rapidly convergent series 
may be obtained by setting large integer $a$. At the same time, it should be noted that precisely for large $a$, first terms of the series may
unexpectedly grow, but after some term they decrease and the series converges. For instance, taking $a=7$, we have
\begin{eqnarray}
&\gamma\,&\displaystyle 
=\,-3\ln2 - \sum_{n=1}^\infty\frac{\,(-1)^n\,\psi_{n}(7)\,}{\,n\,}\,
=\, -3\ln2 + \frac{15}{2} - \frac{293}{24} + \notag\\[1mm]
&&\displaystyle\qquad\qquad\qquad
+ \frac{1079}{72}  - \ldots 
-\frac{8183}{9\,331\,200} - \frac{530\,113}{4\,790\,016\,000} - \ldots\qquad
\end{eqnarray}
Also, the pattern of the sign is not obvious, but for the large index $n$ all the terms should be negative. 
By the way, adding the series with $a=-\frac12$ to that with $a=1$, we eliminate $\ln2$
and thus get a series with rational terms only for Euler's constant
\begin{eqnarray}
&& \gamma\,=\sum_{n=1}^\infty\frac{\,(-1)^{n+1}\,\,}{\,2n\,}
\Big\{\psi_{n}\big(-\tfrac12\big)+ \psi_{n}(1) \Big\}\,
=\,\frac34 - \frac{11}{96} - \frac{1}{72} - \frac{311}{46\,080} - \qquad \notag\\[1mm]
&&\displaystyle\qquad\qquad\qquad\qquad
- \frac{5}{1152}
- \frac{7291}{2\,322\,432}  - \frac{243}{100\,352} - \frac{14\,462\,317}{7\,431\,782\,400}
- \ldots \qquad
\end{eqnarray}
converging at the same rate as $\,\sum n^{-\nicefrac{3}{2}}\ln^{-1}\!n\,$.
Other choices of $a$ are also possible in order to get series with rational terms only for $\gamma$.\footnote{For a list of the most known
series with the rational terms only for Euler's constant, see e.g.~\cite[p.~379]{iaroslav_09}. From the historical viewpoint
it may also be interesting to note that series (36) from \cite[p.~380]{iaroslav_09}, in its second form,
was also given by F.~Franklin already in 1883 in a paper read at a meeting of the University Mathematical Society \cite{franklin_01}. 
Moreover, this series, of course, may be even much older since it is obtained by a quite elementary technique.}
In fact, it is not difficult to show that we can eliminate the logarithm by properly choosing $a$, namely
\be\label{093fujen4t}
\gamma\,=\sum_{n=1}^\infty\frac{\,(-1)^{n+1}\,\,}{\,2n\,}\Big\{\psi_{n}(a)+ \psi_{n}\Big(-\frac{a}{1+a}\Big) \!\Big\}\,,
\qquad a>-1 \,,
\ee
which, at $\,a\in\mathbbm{Q}\,$, represents a huge family of series with rational terms only for Euler's constant.
This series converges at the same rate as $\,\sum n^{-a-2}\ln^{-1}\!n\,$ for $-1<a\leqslant0$ and
$\,\sum n^{-\frac{a+2}{a+1}}\ln^{-1}\!n\,$ for $a\geqslant0$, except for the integer values of $a$ for which $\,\ln^{-1}\!n\,$
should be replaced by $\,\ln^{-2}\!n\,$. In other words, the rate of convergence of \eqref{093fujen4t} 
cannot be worse than
$\,\sum n^{-1-\varepsilon}\ln^{-1}\!n\,$, where $\varepsilon$ is a positive however small parameter, and better than 
$\,\sum n^{-2}\ln^{-2}\!n\,$.
More generally, from \eqref{8769f657fde6} it follows that if $\,a_1,\ldots, a_k\,$ 
and $m$ are chosen so that  $\,(1+a_1)_m\cdots(1+a_k)_m=1\,$, then
\be
\gamma\,=\,\frac{1}{\,m\,k\,}\!\sum_{n=1}^\infty\frac{\,(-1)^{n+1}\,\,}{\,n\,}\sum_{l=1}^k N_{n,m}(a_l)\,,
\qquad a_1,\ldots\,, a_k>-1 \,.
\ee
Furthermore, if for some $\,q_1,\ldots, q_k\,$ and $m$, the quantities 
$\,a_1,\ldots, a_k\,$ are chosen so that  $\,(1+a_1)^{q_1}_m\cdots(1+a_k)^{q_k}_m=1\,$,
then we have a more general formula
\be
\gamma\,=\,\frac{1}{\,m\,(q_1+\ldots+q_k)\,}
\!\sum_{n=1}^\infty\frac{\,(-1)^{n+1}\,\,}{\,n\,}\sum_{l=1}^k q_l\, N_{n,m}(a_l)\,,
\qquad a_1,\ldots\,, a_k>-1 \,,
\ee
which is the most complete 
generalization of the Fontana--Mascheroni series \eqref{09384rtjh3}.

Analogously, one can obtain the series expansions for $\gamma$ from \eqref{kih2g3o8dgg2ikhbde}. 
Indeed, putting for simplicity $v=1$ in the latter expression yields
\begin{eqnarray}\notag 
&&\gamma\,= -\frac{2}{\,m+2a\,}
\left\{\ln\Gamma(a+1) -\frac12\ln2\pi + \frac12 
+ \frac{1}{\,m\,}\!\sum_{n=1}^{m-1} (m-n)\ln(a+n) \, + \right. \qquad\qquad\\[1mm]
&&\qquad\qquad\qquad\qquad\qquad\qquad\qquad\qquad\qquad
\left.+ \, \frac{1}{\,m\,}\!\sum_{n=1}^\infty\frac{(-1)^n\,N_{n+1,m}(a)\,}{n}\right\}\,,
\end{eqnarray}
where $m\in\mathbbm{N}$ and $\Re a>-1$. For $m=1$ the sum in the middle should be taken as zero, so that
\be\label{796gb876v}
\gamma\,= -\frac{2}{\,1+2a\,}
\left\{\ln\Gamma(a+1) -\frac12\ln2\pi + \frac12 +  
\sum_{n=1}^\infty\frac{(-1)^n\,\psi_{n+1}(a)\,}{n}\right\}\,,\qquad \Re a>-1\,.
\ee
At integer and demi--integer values of $a$, we have a quite simple expression for Euler's constant $\gamma$,
which does not contain the $\Gamma$-function; conversely, it may also be regarded as yet another series
for $\ln\Gamma(z)$.
Another interesting consequence of this formula is that it 
readily permits to obtain an interesting series for the digamma function. By multiplying both sides of this formula by $\,m+2a\,$
and by calculating the derivative of the resulting expression with respect to $a$, see \eqref{fhqgg6754r}, we obtain 
\be
\Psi(a+1)\,=\, -\gamma -   
\,\frac{1}{\,m\,}\!\sum_{n=1}^{m-1}\!\frac{\,m-n\,}{a+n}\,-
\,\frac{1}{\,m\,}\!\sum_{n=1}^\infty\frac{(-1)^n\,}{n}
\left\{\!\binom{a+m}{n+1}-\binom{a}{n+1}\!\right\}\,,\qquad
\ee
where $m=2,3,4,\ldots\,$ and $\Re a >-1$.
At $m=1$, we directly get the result from \eqref{796gb876v}
\be
\Psi(a+1)\,=\, -\gamma -   
\sum_{n=1}^\infty\frac{(-1)^n\,}{n}\frac{d\psi_{n+1}(a)}{da}\,=\, -\gamma -   
\sum_{n=1}^\infty\frac{(-1)^n\,}{n}\binom{a}{n}\,,\qquad \Re a >-1\,,
\ee
the series which is sometimes attributed to Stern \cite[p.~251]{norlund_02}.

It is also possible to deduce series expansions with rational terms for Euler's constant from series \eqref{h9843dh2edf}. 
Putting $v=1$, we obtain the following series
\begin{eqnarray}\notag 
&&\label{4098u4rfj3}
\gamma\,=\, \ln2\pi - 2 - 2\!\sum_{n=1}^\infty\frac{\,(-1)^n\,G_{n}^{(2)}\,}{\,n\,}\,=
\, \ln2\pi - 2 + \frac{2}{3} + \\
&& \qquad\qquad\qquad\qquad\qquad
+ \frac{1}{24}+ \frac{7}{540} + \frac{17}{2880}+ \frac{41}{12\,600} 
+ \frac{731}{362\,880} + \ldots \qquad
\end{eqnarray}
converging at the same rate as $\sum n^{-2}\ln^{-3} n$ (see Nota Bene on p.~\pageref{hg8fgvikhvgc6}). 
It is interesting that the numerators of this series, except its first term, coincide with those of the third row of the inverse
Akiyama-Tanigawa algorithm from $1/n$, see OEIS A193546 while its denominator does not seem to be known to the OEIS.
In fact, the reader may easily verify that all these series are new, and at the moment of writing of this paper were not
known to the OEIS [except \eqref{4098u4rfj3}].

\begin{corollary}
The generalized Maclaurin coefficients $\delta_m(v)$ of the regular function $\,\zeta(s,v)-(s-1)^{-1}\,$
admit the following series representations
\begin{eqnarray}\label{hgce8924h}
&\displaystyle\delta_m(v)\, =\,f_m(v)+\sum_{n=0}^\infty\big| G_{n+1}\big
|\sum_{k=0}^{n} (-1)^k \binom{n}{k}\ln^m (k+v) \,, \\[3mm]
\displaystyle
&\displaystyle\delta_m(v)\,=\,f_m(v-1) -\sum_{n=0}^\infty C_{n+1}
\sum_{k=0}^{n} (-1)^k \binom{n}{k}\ln^m (k+v)\,, \label{jhc20n3d2d} \\[3mm]
\displaystyle
&\displaystyle\delta_m(v)\,=\, \frac{1}{\,r\,}\!\sum_{l=0}^{r-1} f_m(v+a+l) + 
\frac{1}{\,r\,}\!\sum_{n=0}^\infty (-1)^n\,N_{n+1,r}(a)
\sum_{k=0}^{n} (-1)^k \binom{n}{k}\ln^m (k+v)\,, \notag \\[-3mm]
\label{lkjwehf3890b}
\end{eqnarray}
where we denoted
\be\notag
f_m(v)\equiv(-1)^m m!\left\{1-v-v\sum_{k=1}^m(-1)^k\frac{\ln^k v}{k!} \right\}
\ee
for brevity. 
\end{corollary}

\begin{proof}
Generalizing expansion \eqref{uyh0hf5} to the Hurwitz $\zeta$-function,
we may introduce $\delta_m(v)$, $m\in\mathbbm{N}$, 
$v\in\mathbbm{C}\setminus\!\{0,-1,-2,\ldots\}$, as the coefficients in the expansion
\be
\label{hgx23g}
\zeta(s,v)\,=\,\frac{1}{\,s-1\,} + \frac{3}{2} - v + \sum_{m=1}^\infty\frac{(-1)^m\delta_m(v)
}{m!}\,s^m \,, \qquad s\in\mathbbm{C}\setminus\!\{1\}\,.
\ee
It is, therefore, not difficult to see that $\,\delta_m(v)=(-1)^m\big\{\zeta^{(m)}(0,v) +m!\big\}\,$.
The desired formul\ae~are obtained by a direct differentiation of \eqref{jxh293nbwes}--\eqref{98y6b796}
respectively. 
Similarly to the generalized Stieltjes constants, functions $\delta_m(v)$ enjoy 
a recurrence relation $\,\delta_m(v+1)=\delta_m(v)-\ln^m \!v\,$, which may be used
to rewrite \eqref{hgce8924h}--\eqref{jhc20n3d2d} in a slightly different form if necessary. 
\end{proof}

\vspace{1em}
\noindent{\bfseries Remark 3, related to the logarithm of the ${\boldsymbol\Gamma}$-function.}
Recalling that $\,\delta_1(v)=-\ln\Gamma(v) + \frac12\ln2\pi-1\,$ and 
noticing that $\,f_1(v)=v-1-v\ln{v}\,$, gives us these three series expansions
\be\label{b45wthhsa}	
\ln\Gamma(v)\, =\,v\ln v - v + \frac12\ln2\pi - \sum_{n=0}^\infty\big| G_{n+1}\big
|\sum_{k=0}^{n} (-1)^k \binom{n}{k}\ln(k+v) \,, 
\ee
where $\,\Re v>0\,$,
\be\label{g34tg66g} 	
\ln\Gamma(v)\,=\, (v-1)\ln(v-1) -v+1 + \frac12\ln2\pi  + \sum_{n=0}^\infty C_{n+1}
\sum_{k=0}^{n} (-1)^k \binom{n}{k}\ln(k+v) \,,   \quad
\ee
where $\,\Re v>1\,$, and 
\begin{eqnarray}
&&\displaystyle\ln\Gamma(v)\,=\, \frac{1}{\,r\,}\!\sum_{l=0}^{r-1} (v+a+l)\ln(v+a+l) 
- v - a - \frac{r}{2}+ \frac12\ln2\pi + \qquad\label{893y4hfpdf}\\[1mm]
&&\displaystyle\qquad\qquad\qquad
  + \frac12 - \frac{1}{\,r\,}\!\sum_{n=0}^\infty (-1)^n\,N_{n+1,r}(a)
\sum_{k=0}^{n} (-1)^k \binom{n}{k}\ln(k+v) \,,\qquad\notag
\end{eqnarray}
$\Re v>-\Re{a}\,$, $\Re a >-1$, $\,r\in\mathbbm{N}\,$
for the logarithm of the $\Gamma$-function. 
First of these representations is equivalent to 
a little-known formula for the logarithm of the $\Gamma$-function, which appears in epistolary exchanges between Charles Hermite and Salvatore Pincherle dating back to 1900
\cite[p.~63, two last formul\ae]{hermite_01}, \cite[vol.~IV, p.~535, third and fourth formul\ae]{hermite_02}, 
while the second and the third representations seem to be novel. 

Note also that the parameter $a$, in all the expansions 
in which it appears, plays the role of the ``rate of convergence'': the greater this parameter,
the faster the convergence, especially if $a$ is an integer.

\vspace{1em}
\noindent{\bfseries Remark 4, related to some similar expansions containing the same finite differences.}\label{g897tvt6y}
It seems appropriate to note that there exist other expansions of the same nature, which merit to be mentioned here. 
For instance
\begin{eqnarray}
&&\hspace{-5mm}\ln\Gamma(v+x)\,=\,\ln\Gamma(v) + \sum_{n=0}^\infty (-1)^n\binom{x}{n+1}\!
\sum_{k=0}^n (-1)^k \binom{n}{k}\ln(k+v) \,,\label{hgd28}\\[3mm]
&&\hspace{-5mm}\ln\Gamma(v)\,=\,-v+\frac{1}{2}+\frac{1}{2}\ln2\pi +\sum_{n=0}^\infty\frac{1}{\,n+1\,}
\!\sum_{k=0}^n (-1)^k \binom{n}{k}(k+v)\ln(k+v) \,, \qquad\\[3mm]
&&\hspace{-5mm}\Psi(v)\,=\,\sum_{n=0}^\infty\frac{1}{\,n+1\,}
\!\sum_{k=0}^n (-1)^k \binom{n}{k}\ln(k+v) \,, \label{348rfy3hfbf3b}\\[3mm]
&&\hspace{-5mm}\gamma_m(v)\,=\,-\frac{1}{\,m+1\,}\!\sum_{n=0}^\infty\frac{1}{\,n+1\,}
\!\sum_{k=0}^n (-1)^k \binom{n}{k}\ln^{m+1}(k+v)\,,
\qquad m\in\mathbbm{N} \,,
\end{eqnarray}
see e.g.~\cite[p.~59]{hermite_01}, \cite[vol.~IV, p.~531]{hermite_02}, \cite[p.~251]{norlund_02}, 
\cite{connon_06}, \cite{connon_09}.
The three latter formul\ae~are usually deduced from Hasse's series \eqref{jnd2h93ndd}, but they equally may be obtained from a more general formula
\be\label{kjh982hed}
\frac{dF(v)}{dv}\,=\sum_{n=0}^\infty \!\frac{(-1)^n}{\,n+1\,}\,\Delta^{n+1} F(v)\,
=-\sum_{n=1}^\infty \!\frac{1}{\,n\,}\sum_{k=0}^{n} (-1)^k \binom{n}{k}F(v+k)\,,
\ee
which is a well--known result in the theory of finite differences,
see e.g.~\cite[pp.~240--242]{norlund_02}. Moreover, Hasse's series itself \eqref{jnd2h93ndd} is a simple consequence
of this formula. Putting $F(v)=\zeta(s,v)$, we actually have
\be
\Delta\zeta(s,v)=\zeta(s,v+1)-\zeta(s,v)=-v^{-s}\quad \text{and}\quad 
\frac{\,\partial\zeta(s,v)\,}{\partial v}=-s\,\zeta(s+1,v)\,.\quad
\ee
Hence, \eqref{kjh982hed} reads
\be
s\,\zeta(s+1,v)\,=\sum_{n=0}^\infty \!\frac{(-1)^n}{\,n+1\,}\,\Delta^{n} v^{-s}
=\sum_{n=0}^\infty \!\frac{(-1)^n}{\,n+1\,}\sum_{k=0}^{n} (-1)^k \binom{n}{k}(v+k)^{-s}.
\ee
Rewriting this expression for $s-1$ instead of $s$ and then dividing both sides by $s-1$
immediately yield Hasse's series \eqref{jnd2h93ndd}.

It is also evident that expressions containing the fractions $\,\frac{1}{n+1}\,$ are more simple than those
with the coefficients $G_n$, $C_n$ or $N_{n,m}(a)$. Applying \eqref{kjh982hed} to the Dirichlet series $\zeta(s,v,\boldsymbol{u})$
introduced in Theorem \ref{ujihd97823hdbe} and proceeding analogously, we may obtain without difficulty the following result
\be
\zeta(s,v,\boldsymbol{u})\,=\,\frac{1}{\,s-1\,}\!
\sum_{n= 0}^\infty \frac{(-1)^n}{\,n+1\,}\, \Delta^{n}\zeta(s-1,v+1, \Delta \boldsymbol{u})
+ \frac{u_0}{\,s-1\,}\!\sum_{n=0}^\infty \frac{(-1)^n}{\,n+1\,}\, \Delta^{n} v^{1-s}   
\ee
which is the analog of  \eqref{73dyd2odhedy}. Considering again $\zeta(s,v,\boldsymbol{u})$
with $u_n=1+n$, we easily derive
the following functional relationship 
\be\label{gh87tgkgft}
(v-1)\,\zeta(s,v) \,=\,
\frac{s-2}{\,s-1\,}\,\zeta(s-1,v)\,
- \frac{1}{\,s-1\,}\!\sum_{n=0}^\infty \frac{(-1)^n}{\,n+2\,} \,\Delta^{n} v^{1-s} \,,
\qquad s\in\mathbbm{C}\setminus\!\{1,2\}\,,
\ee
or explicitly
\be\label{jhdf29h3r2}
\zeta(s,v) =\frac{1}{(v-1)\,(s-1)}\!
\left\{\! (s-2)\,\zeta(s-1,v)
- \!\sum_{n=0}^\infty \!\frac{1}{\,n+2\,} \! \sum_{k=0}^n (-1)^k
\binom{n}{k}(k+v)^{1-s}\!\right\}\!,
\ee
$ s\in\mathbbm{C}\setminus\!\{1,2\}\,,$
which is, in some senses, analogous to \eqref{9u8y943dh34hb3} and which is  
yet another generalization of Ser's formula \eqref{hdf398hd}.\footnote{We precisely obtain the latter
as a particular case of \eqref{gh87tgkgft} when $v=1$.} Another functional relationship for $\zeta(s)$ may be obtained
by setting $v=\frac12$, because $\zeta\big(s,\frac12\big)=\big(2^s-1\big)\zeta(s)$ [this obviously 
also applies to other formul\ae~we obtained for $\zeta(s,v)$]. 
Now, multiply both sides of the latter equality by $s-1$ and expand it into the Taylor series about $(s-1)$.\footnote{The $\zeta$-functions
should be expanded into the Laurent series accordingly to \eqref{dhd73vj6s1b}.} 
Equating the coefficients of $(s-1)^1$ yields the following relationship between the gamma and the digamma functions
\be\label{84rfh3nbd}
\Psi(v)\,=\,\frac{1}{\,v-1\,}\!
\left\{\ln\Gamma(v) + v -\frac12\ln2\pi -\frac12 -
\sum_{n=0}^\infty \frac{1}{\,n+2\,} \sum_{k=0}^n (-1)^k \binom{n}{k}\ln(k+v)\right\}.
\ee
Similarly, by equating the coefficients of $(s-1)^{m+1}$ we have for the generalized Stieltjes constants 
\begin{eqnarray}
&&\displaystyle \gamma_m(v)\,=\,\frac{1}{\,v-1\,}
\left\{(-1)^m \zeta^{(m)}(0,v) - \frac{ (-1)^m}{\,m+1\,}\,\zeta^{(m+1)}(0,v)\, 
+  \right.\label{f786rc87fcoi}\\[2mm]
&&\displaystyle \qquad \qquad \qquad \qquad \qquad
\left. + \,\frac{1}{\,m+1\,}\!\sum_{n=0}^\infty \frac{1}{\,n+2\,}
\sum_{k=0}^{n} (-1)^k \binom{n}{k}\ln^{m+1} (k+v) \right\}\,,\notag
\end{eqnarray}
where $m\in\mathbbm{N}$, $\Re v >0$ and where the right part should be regarded as a limit when $v=1$.\footnote{It is interesting to compare this formula with \eqref{ilu42798hd2ibd}.}

Furthermore, proceeding similarly and using formula (67) from \cite[p.~241]{norlund_02}, 
one may obtain the following expansions
\begin{eqnarray}
&&\hspace{-12mm}\zeta(s)\,=\,\frac{1}{\,s-1\,}\sum_{n=0}^\infty H_{n+1}\!\sum
_{k=0}^n (-1)^k \binom{n}{k}(k+2)^{1-s} \,,\\[3mm]
&&\hspace{-12mm}\zeta(s,v-1)\,=\,\frac{1}{\,s-1\,}\sum_{n=0}^\infty H_{n+1}\!\sum
_{k=0}^n (-1)^k \binom{n}{k}(k+v)^{1-s} \,,\\[3mm]
&&\hspace{-12mm}\ln\Gamma(v-1)\,=\,-v+\frac{3}{2}+\frac{1}{2}\ln2\pi +\sum_{n=0}^\infty  H_{n+1}
\!\sum_{k=0}^n (-1)^k \binom{n}{k}(k+v)\ln(k+v) \,,\\[3mm]
&&\hspace{-12mm}\Psi(v-1)\,=\,\sum_{n=0}^\infty H_{n+1}
\!\sum_{k=0}^n (-1)^k \binom{n}{k}\ln(k+v)  \label{98y43fh32onrf4} \,,\\[3mm]
&&\hspace{-12mm}\gamma\,=\,-\!\sum_{n=0}^\infty H_{n+1}
\!\sum_{k=0}^n (-1)^k \binom{n}{k}\ln(k+2) \,, \\[3mm]
&&\hspace{-12mm}\gamma_m\,=\,-\frac{1}{\,m+1\,}\!\sum_{n=0}^\infty H_{n+1}
\!\sum_{k=0}^n (-1)^k \binom{n}{k}\ln^{m+1}(k+2) \,,\\[3mm]
&&\hspace{-12mm}\gamma_m(v-1)\,=\,-\frac{1}{\,m+1\,}\!\sum_{n=0}^\infty H_{n+1}
\!\sum_{k=0}^n (-1)^k \binom{n}{k}\ln^{m+1}(k+v) \,,
\end{eqnarray}
where $\,\Re v>1\,$, $\,m\,$ is a natural number and 
$\,H_n=1+\frac12+\frac13+\ldots+\frac1n\,$ is the $n$th harmonic number.
Using the corresponding recurrence relations, these
formul\ae~may also be written for $v$ instead of $v-1$ in the left part. 
Besides, for the Dirichlet series $\zeta(s,v,\boldsymbol{u})$ from Theorem \ref{ujihd97823hdbe}, we also have
\begin{eqnarray}
&& 
\zeta(s,v-1,\boldsymbol{u})\,=\,\frac{1}{\,s-1\,}\!
\sum_{n= 0}^\infty (-1)^n H_{n+1} \, \Delta^{n}\zeta(s-1,v+1, \Delta \boldsymbol{u}) \,+ \notag\\
&&\qquad\qquad\qquad\qquad\qquad\qquad\qquad\qquad\qquad \qquad \notag
+ \frac{u_0}{\,s-1\,}\!\sum_{n=0}^\infty (-1)^n H_{n+1} \, \Delta^{n} v^{1-s} .  \qquad \qquad 
\end{eqnarray}
If, for example, $\,u_n=n+1\,$, then
\begin{eqnarray}
&& \zeta(s,v-1) \,=\,\frac{1}{\,(v-2)\,(s-1)\,}
\left\{\vphantom{\sum_{n=0}^\infty} 
(v-1)^{1-s} +\,
(s-2)\,\zeta(s-1,v-1)\,-   \right. \qquad\qquad\qquad\label{893u4f3jmncf4}\\[1mm]
&&\qquad\qquad\qquad\qquad\qquad\qquad\qquad\qquad\notag
\left. - \sum_{n=0}^\infty H_{n+2} \sum_{k=0}^n (-1)^k \binom{n}{k}(k+v)^{1-s}\right\}\,,\qquad
\end{eqnarray}
$s\in\mathbbm{C}\setminus\!\{1,2\}$, $\Re v>1$. By expanding 
both sides of this formula into the Laurent series about $s=1$, we get
\begin{eqnarray}
&&\Psi(v-1)\,=\,\frac{1}{\,v-2\,}
\left\{ \ln\Gamma(v)+v-\frac{3}{2}-\frac{1}{2}\ln2\pi - \vphantom{\sum_{k=0}^n (-1)^k} \right.\\[1mm]
&&\displaystyle\qquad\qquad\qquad\qquad\qquad\qquad\qquad \notag
\left.  - \sum_{n=0}^\infty  H_{n+2}
\!\sum_{k=0}^n (-1)^k \binom{n}{k}\ln(k+v)\right\},\\[3mm]
&&\displaystyle \gamma_m(v-1)\,=\,\frac{1}{\,v-2\,}
\left\{-\frac{\,\ln^{m+1}(v-1)\,}{m+1} \,+\, (-1)^m \zeta^{(m)}(0,v-1) \, - \, 
  \right.\label{f786rc87fcoi}\\[2mm]
&&\displaystyle \quad 
\left. - \,\frac{ (-1)^m}{\,m+1\,}\,\zeta^{(m+1)}(0,v-1) \, + \,\frac{1}{\,m+1\,}\!\sum_{n=0}^\infty H_{\,n+2\,}
\sum_{k=0}^{n} (-1)^k \binom{n}{k}\ln^{m+1} (k+v) \right\}\notag
\end{eqnarray}
holding for $\Re v>1$.
If, in addition, we set $\,v=2\,$ in \eqref{893u4f3jmncf4}, we get
\be
\zeta(s)\,=\,\frac{1}{\,s-1\,}\left\{-1 + \sum_{n=0}^\infty H_{n+2} \sum
_{k=0}^n (-1)^k \binom{n}{k}(k+2)^{-s}\right\}\,,\qquad s\in\mathbbm{C}\setminus\!\{1\}\,,
\ee
and hence
\begin{eqnarray}
&&\gamma\,=\,-\!\sum_{n=0}^\infty H_{n+2}
\!\sum_{k=0}^n (-1)^k \binom{n}{k}\frac{\,\ln(k+2)\,}{k+2}\,,\\[3mm]
&&\gamma_m\,=\,-\frac{1}{\,m+1\,}\!\sum_{n=0}^\infty H_{n+2}
\!\sum_{k=0}^n (-1)^k \binom{n}{k}\frac{\,\ln^{m+1}(k+2)\,}{k+2}\,,
\qquad m\in\mathbbm{N} \,.
\end{eqnarray}

Finally, differentiating \eqref{84rfh3nbd} with respect to $v$ and simplifying the sum with the binomial coefficients yields a formula for the
trigamma function $\Psi_1(v)$
\be\label{opi8rj34r}
\Psi_1(v)\,=\,\frac{1}{\,v-1\,}\left\{1-
\sum_{n=0}^\infty \frac{n!}{\,(n+2)\cdot (v)_{n+1}\,} \right\}\,,\quad \Re v>0\,,
\ee
where for $v=1$ the right part should be calculated via an appropriated limiting procedure.
Similarly, from \eqref{98y43fh32onrf4} we get 
\be
\Psi_1(v-1)\,=
\sum_{n=0}^\infty \frac{\,H_{n+1}\,n!\,}{\,(v)_{n+1}\,} \,,\quad \Re v>1\,.
\ee
which may also be written as 
\be
\Psi_1(v)\,=
\,v\! \sum_{n=0}^\infty \frac{\,H_{n+1}\,n!\,}{\,(v)_{n+2}\,} \,,\quad \Re v>0\,,
\ee
if we put $v+1$ instead of $v$.
These formul\ae~may be compared with
\be
\Psi_1(v)\,=
\sum_{n=0}^\infty \frac{n!}{\,(n+1)\cdot (v)_{n+1}\,} \,,\quad \Re v>0\,,
\ee
which may be readily obtained from \eqref{348rfy3hfbf3b} by following the same line of reasoning.

\begin{theorem}\footnote{We place this theorem
after the other results because, on the one hand, it logically continues the previous remark, and on the other hand, this result merits to be
presented as a separate theorem, rather than a simple formula in the text.}
The function $\zeta(s,v)$ may be represented by the following series with the Stirling numbers 
of the first kind
\begin{eqnarray}
&\zeta(s,v)& = \,\frac{k!}{\,(s-k)_k\,}   \label{94fnm4fonk3e}
\!\sum_{n=0}^\infty \frac{\,\big|S_1(n+k,k)\big|\,}{(n+k)!}\!\!\sum_{l=0}^{n+k-1}\! (-1)^l \binom{n+k-1}{l} (l+v)^{k-s} \\[2mm]
&&
= \,\frac{k!}{\,(s-k)_k\,}   
\!\sum_{n=0}^\infty \frac{\,\big|S_1(n+k,k)\big|\,}{(n+k)!}\!
\sum_{r=0}^{k-1}(-1)^r \binom {k-1}{r}
\sum_{l=0}^{n}\! (-1)^{l} \binom{n}{l} (l+r+v)^{k-s}  \notag
\end{eqnarray}
where $k\in\mathbbm{N}$.
It may also be written in terms of the more simple numbers, especially for small $k$, see \eqref{uf8peorkjg4e0jgmfg}
hereafter. For example, for $k=1$, we obtain Hasse's series \eqref{jnd2h93ndd}; for $k=2$, we have
\begin{eqnarray}
&\zeta(s,v)& =\,\frac{\, 2\,}{\,(s-1)(s-2)\,} 
\!\sum_{n=0}^\infty \!\frac{\,H_{n+1}\,}{\,n+2\,}
\!\sum_{l=0}^{n+1} (-1)^l \binom{n+1}{l} (l+v)^{2-s} \label{erve8g}\\[2mm]
&& =\,\frac{\, 2\,}{\,(s-1)(s-2)\,} \notag
\!\sum_{n=0}^\infty \!\frac{\,H_{n+1}\,}{\,n+2\,}
\!\sum_{l=0}^{n} (-1)^l \binom{n}{l}
\Big\{(l+v)^{2-s} - (l+v+1)^{2-s}\Big\}\,,
\end{eqnarray}
for $k=3$
\be\notag
\zeta(s,v)\, =\,\frac{\, 3\,}{\,(s-1)(s-2)(s-3)\,} 
\!\sum_{n=0}^\infty \!\dfrac{\,H^2_{n+2} - H^{(2)}_{n+2}\,}{n+3}
\!\sum_{l=0}^{n+2} (-1)^l \binom{n+2}{l} (l+v)^{3-s},
\ee
and so on, where $H^{(s)}_{n}=1^{-s}+2^{-s}+3^{-s}+\ldots+n^{-s}$ are the generalized
harmonic numbers, also known as the incomplete $\zeta$-function. \label{ufh3984fbnf}
\end{theorem}

\begin{proof}
From the theory of finite differences, it is known that
\be\label{987ybh789y}
\frac{d^kF(v)}{dv^k}\,=\,k\!\sum_{n=0}^\infty \!\frac{B^{(n+k)}_n}{\,(n+k)\,n!\,}\,\Delta^{n+k} F(v)\,,
\qquad k\in\mathbbm{N}\,,
\ee
where $\,B^{(n+k)}_n$ are the Bernoulli numbers of higher order,  
see e.g.~\cite[p.~242]{norlund_02}.\footnote{This
formula is actually a generalization of \eqref{kjh982hed}, which we get when $k=1$.}
These numbers are the particular case of the Bernoulli polynomials
of higher order $\,B^{(r)}_n=B^{(r)}_n(0)$, see \eqref{j023jd3ndu}, and for positive integers $r>n$ 
they can be expressed in terms of the Stirling numbers of the first kind. 
On the one hand, we have
\be
\left\{\!\frac{\,\ln(1+z)\,}{z}\!\right\}^{\!s}\,=
\sum_{n=0}^\infty \frac{z^n}{n!}\,\frac{s}{s+n}\,B^{(n+s)}_n\,,\qquad |z|<1\,,\quad s\in\mathbbm{C}\,,
\ee
see e.g.~\cite[p.~244]{norlund_02}, \cite[p.~135]{milne_01}; on the other hand, 
the generating equation for the Stirling numbers of the first kind reads\footnotemark[9]
\be
\frac{\ln^s(1+z)}{s!}\,=\sum_{n=s}^\infty\!\frac{S_1(n,s)}{n!}z^n 
\,=\sum_{n=0}^\infty\!\frac{S_1(n+s,s)}{(n+s)!}z^{n+s}\,, \qquad s\in\mathbbm{N}_0\,.
\ee
Hence, if $s$ is a nonnegative integer, we may equate the coefficients of $z^{n+s}$ in the right parts of two latter formul\ae. 
This yields
\be
B^{(n+s)}_n\,
=\,\frac{\,(s-1)!\, S_1(n+s,s)\,}{\,(n+1)(n+2)\cdots(n+s-1)\,}\,
=\,\frac{\,n!\,(s-1)!\,}{\,(n+s-1)!\,}\, S_1(n+s,s)\,,
\ee
where $n$ and $s$ are both natural numbers, or equivalently
\be
B^{(r)}_n\,=\,\frac{\,n!\,(r-n-1)!\, }{\,(r-1)!\,}\, S_1(r,r-n)\,,\qquad
n,r\in\mathbbm{N}\,, \quad r>n\,.
\ee
Formula \eqref{987ybh789y}, therefore, takes the form
\be\label{hg7yvf42quk}
\frac{d^kF(v)}{dv^k}\,=\,k!\!\sum_{n=0}^\infty \!\frac{S_1(n+k,k)\,}{\,(n+k)!\,}\,\Delta^{n+k} F(v)\,,
\qquad k\in\mathbbm{N}\,,
\ee
where, it is appropriate to remark that 
\be\label{uf8peorkjg4e0jgmfg}
\frac{\,\big|S_1(n+k,k)\big|\,}{\,(n+k)!\,}\,=
\begin{cases}
\,\dfrac{1}{\,n+1\,} \,,  & k=1\\[3mm]
\,\dfrac{\,H_{n+1}}{n+2}\,,& k=2 \\[3mm]
\,\dfrac{\,H^2_{n+2} - H^{(2)}_{n+2}\,}{\,2\,(n+3)} \,,\qquad\quad& k=3 \\[3mm]
\,\ldots\ldots\ldots\ldots\ldots\ldots\ldots\ldots\ldots\ldots\ldots\ldots   \\[3mm]
\,\dfrac{\,P_{k-1}\Big(H_{n+k-1}^{(1)}, -H_{n+k-1}^{(2)},\ldots, 
(-1)^{k}H_{n+k-1}^{(k-1)}\Big)\,}{\,n+k\,} \,,\qquad\quad& k\in\mathbbm{N} \\[3mm] 
\end{cases}
\ee
where $P_{k-1}$ are the modified Bell polynomials
defined by the following generating function
\be
\exp\!\left(\sum\limits_{n=1}^{\infty}x_n \frac{z^n}{n}\right) = 
\sum\limits_{m=0}^{\infty}P_m(x_1,\cdots,x_m)\, z^m \,,
\ee
see e.g.~\cite[p.~217]{comtet_01}, \cite{flajolet_01}, \cite{candelpergher_01}, \cite{candelpergher_05}. Now,
the repeated deifferentiation of $\zeta(s,v)$ yields
\be
\frac{\,\partial^k\zeta(s,v)\,}{\partial v^k}=\,(-1)^k (s)_k\,\zeta(s+k,v)\,,\qquad k\in\mathbbm{N}\,,
\ee
and since $\Delta\zeta(s,v)=\zeta(s,v+1)-\zeta(s,v)=-v^{-s}$,
formula \eqref{hg7yvf42quk} becomes
\be\label{hg7yvf42quk}
\zeta(s+k,v)\,=\,\frac{\,(-1)^{k-1} k!\,}{(s)_k} 
\!\sum_{n=0}^\infty \!\frac{S_1(n+k,k)\,}{\,(n+k)!\,}\,\Delta^{n+k-1} v^{-s}\,,
\qquad k\in\mathbbm{N}\,.
\ee
Rewriting this expression for $s$ instead of $s+k$ and using 
the fact $\,\sgn S_1(m,n)=(-1)^{m\pm n}\,$, we arrive at
\be\label{erve3g}
\zeta(s,v)\,=\,\frac{\, k!\,}{\,(s-k)_k\,} 
\!\sum_{n=0}^\infty \!\frac{\,\big|S_1(n+k,k)\big|\,}{\,(n+k)!\,}\,(-1)^{n+k-1}\Delta^{n+k-1} v^{k-s}\,,
\qquad k\in\mathbbm{N}\,,
\ee
which is identical with \eqref{94fnm4fonk3e}.
The latter formula may also be written in terms of the $n$th finite difference instead of the $(n+k-1)$th one.
Indeed, since $\Delta^{n+k-1} v^{k-s}=\Delta^{k-1}\big(\Delta^{n} v^{k-s}\big)$, then
by virtue of \eqref{0i3u40jfmnr}, we have
\be\label{erve7g}
\zeta(s,v)\,=\,\frac{\,  k!\,}{\,(s-k)_k\,} 
\!\sum_{n=0}^\infty \!\frac{\,\big|S_1(n+k,k)\big|\,}{\,(n+k)!\,}\!
\sum_{r=0}^{k-1}(-1)^r \binom {k-1}{r}(-1)^{n}\Delta^{n} (v+r)^{k-s}\,,
\ee
where $\,k\in\mathbbm{N}.$
Particular cases of the above formul\ae~may be of interest.
For instance, if $k=1$, then by virtue of \eqref{uf8peorkjg4e0jgmfg} the latter
formula immediately reduces to Hasse's series \eqref{jnd2h93ndd}. If $k=2$, then
\begin{eqnarray}
&\zeta(s,v)& =\,\frac{\, 2\,}{\,(s-1)(s-2)\,} 
\!\sum_{n=0}^\infty \!\frac{\,H_{n+1}\,}{\,n+2\,}\,(-1)^{n+1} \Delta^{n+1} v^{2-s}  \label{erve8g}\\[2mm]
&& =\,\frac{\, 2\,}{\,(s-1)(s-2)\,} \notag
\!\sum_{n=0}^\infty \!\frac{\,H_{n+1}\,}{\,n+2\,}\,(-1)^{n}
\Big\{\Delta^{n} v^{2-s} - \Delta^{n} (v+1)^{2-s}\Big\}\,.
\end{eqnarray}
For $k=3$, we get
\begin{eqnarray}
&\zeta(s,v)& =\,\frac{\, 3\,}{\,(s-1)(s-2)(s-3)\,} 
\!\sum_{n=0}^\infty \!\dfrac{\,H^2_{n+2} - H^{(2)}_{n+2}\,}{n+3}\,
\,(-1)^{n+2} \Delta^{n+2} v^{3-s}  \phantom{mmmmmm}  \label{erve9g}\\[2mm]
&& =\,\frac{\, 3\,}{\,(s-1)(s-2)(s-3)\,} 
\!\sum_{n=0}^\infty \!\dfrac{\,H^2_{n+2} - H^{(2)}_{n+2}\,}{n+3}\,(-1)^{n}
\Big\{\Delta^{n} v^{3-s} - \notag\\[2mm]
&& \qquad\qquad\qquad\qquad\qquad\qquad\qquad\qquad
- 2\Delta^{n} (v+1)^{3-s} + \Delta^{n} (v+2)^{3-s} \Big\}\,\notag
\end{eqnarray}
and so on. 
\end{proof}

The formul\ae~obtained in the above theorem give rise to many interesting expressions. For instance,
putting $v=1$ we get the following series representation for the classic Euler--Riemann $\zeta$-function
\be
\zeta(s)\, = \,\frac{k!}{\,(s-k)_k\,}   \label{94fnm4fonk3e}
\!\sum_{n=0}^\infty \frac{\,\big|S_1(n+k,k)\big|\,}{(n+k)!}\!\!\sum_{l=0}^{n+k-1}\!
(-1)^l \binom{n+k-1}{l} (l+1)^{k-s} \,,
\ee
$k\in\mathbbm{N}$.
It is also interesting that the right of this expression, following the fraction \mbox{$k!/(s-k)_k$}, 
contains zeros of the first order at $\,s=2, 3,\ldots, k\,$, while it has a fixed quantity when $s\to1$. This, perhaps,
may be useful for the study of the Riemann hypothesis, since the right part also contains the zeros 
in the strip $\,0<\Re s <1$. Furthermore, expanding both sides of \eqref{94fnm4fonk3e}
into the Laurent series about $\,s-1$, $s-2,\ldots\,$ one can obtain the series expansions for $\Psi(v)$, 
$\gamma_m(v)$, $\Psi_1(v)$, etc. For example, the Laurent series in a neighborhood of $s=1$ of \eqref{erve8g}
yields
\be\label{c3489fjhn}
\Psi(v) \,=\,-1- 2 
\!\sum_{n=0}^\infty \!\frac{\,H_{n+1}\,}{\,n+2\,}
\sum_{l=0}^{n+1} (-1)^l \binom{n+1}{l} (l+v)\ln(l+v) \,,
\ee
while that in a neighborhood of $s=2$ gives a series for the trigamma function
\be
\Psi_1(v) \,=\,- 2 
\!\sum_{n=0}^\infty \!\frac{\,H_{n+1}\,}{\,n+2\,}
\sum_{l=0}^{n+1} (-1)^l \binom{n+1}{l} \! \ln(l+v) \,,
\ee
which may also be obtained from \eqref{c3489fjhn} by a direct differentiation.

\section*{Acknowledgments} 
The author is grateful to Jacqueline Lorfanfant, Christine Disdier, Alexandra Miric, Francesca Leinardi, Alain Duys, Nico Temme, Olivier Lequeux, 
Khalid Check-Mouhammad and Laurent Roye for providing high-quality scans of several references. 
The author is also grateful to Jacques G\'elinas for his interesting remarks and comments.
Finally, the author owes a debt of gratitude to Donal Connon, who kindly revised the text
from the point of view of English.


\appendix
\section*{Appendix. The integral formula for the Bernoulli polynomials of the second kind 
(Fontana--Bessel polynomials) $\boldsymbol\psi_n(x)$}
Integral representation \eqref{fmiuhgd94} may be obtained by various methods. Below we propose 
a contour integration method, which leads quite rapidly to the desired result. 

Rewrite the generating equation for $\psi_n(x)$, formula \eqref{oi09u03dj}, for $u$ instead of $z$. 
Setting $\,z=-(1+u)=(1+u)e^{+ i \pi }\,$, or equivalently $\,1+u=z\,e^{- i \pi }\,$, 
the latter formula becomes 
\be\notag
\frac{z+1}{\,\ln z-\pi i \,}\,(-z)^x\,=
\frac{(z+1) \, z^x e^{-i\pi x}}{\,\ln z-\pi i \,}\,=
\sum_{n=0}^\infty (-1)^{n-1}\psi_n(x) \, (z+1)^n\,,\qquad |z+1|<1\,.
\ee
where in the right part we replaced $\,e^{- i \pi(n-1) }\,$ by $\,(-1)^{n-1}\,$.
Evaluating now the following line integral along a  
contour $C$ (see Fig.\ \ref{kc30jfd}),
\begin{figure}[!t]    
\centering 
\includegraphics[width=0.5\textwidth]{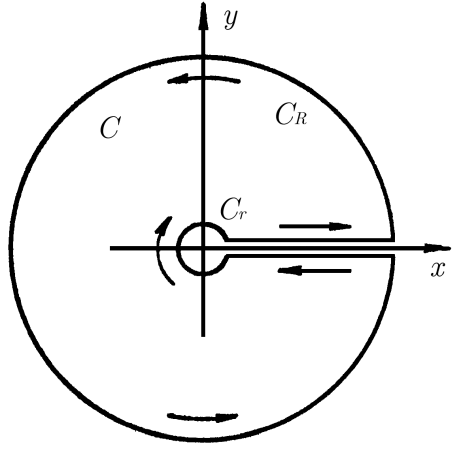} 
\caption{Integration contour $C$ ($r$ and $R$ are radii of the small and big circles respectively, where 
$r\ll1$ and $R\gg1$).} 
\label{kc30jfd} 
\end{figure} 
and then letting $R\to\infty\,$, $r\to0$, we have\footnote{Note that for the clarity we should keep $z^x e^{-i\pi x}$
rather than simply $(-z)^x$.}
\begin{eqnarray} 
&& \displaystyle  
\ointctrclockwise\limits_{C} \frac{\,z^x e^{-i\pi x}\,}{\,(1+z)^n \, (\ln z - \pi i)\,}\,dz\,=\, 
\int\limits_{r}^R \! \ldots  \,\, + \, \int\limits_{C_R} \!\ldots \,\,  
+ \int\limits_{Re^{2i\pi}}^{\;re^{2i\pi}} \!\!\!\!\ldots \,\, +\, \int\limits_{C_r} \! \ldots 
\stackrel{\substack{R\to\infty \\ r\to0}}{=}  \quad    \label{89dyndhed}\\[2mm] 
&& \displaystyle \;\notag
= \int\limits_0^\infty \!\left\{\frac{e^{-\pi i x}}{\ln z - \pi i} - \frac{e^{+\pi i x}}{\ln z + \pi i} \right\}
\cdot\frac{z^x \, dz}{(1+z)^n} \, 
=\,2i\! \int\limits_0^\infty \! \frac{\,\pi \cos\pi x - \sin\pi x \ln z\,}{\,\ln^2 z +\pi^2\,}
\cdot\frac{z^x \, dz}{\,(1+z)^n}
\end{eqnarray} 
since for $\,R\to\infty\,$
\be\notag
\left| \,\int\limits_{C_R} \! \frac{\,z^x\,e^{-i\pi x}\,}{\,(1+z)^n \, (\ln z - \pi i)\,}\,dz\, \right|  
\,=\, O\left(\!\frac{1}{\,R^{n-x-1}\ln R\,}\!\right)=o(1)\,, \qquad  n\geqslant x+ 1\,,
\ee
and for $\,r\to0\,$
\be\notag
\left| \,\int\limits_{C_r} \! \frac{\,z^x\,e^{-i\pi x}\,}{\,(1+z)^n \, (\ln z - \pi i)\,}\,dz \, \right|  
\,=\, O\left(\!\frac{\,r^{x+1}\,}{\,\ln r\,}\!\right)=o(1)\,, \qquad x\geqslant -1\,,   
\ee
where $n\in\mathbb{N}$ and $x$ is assumed to be real.
But the above contour integral may also be evaluated by means of the Cauchy residue theorem.
Since the integrand has only two singularities, a branch point at $z=0$ which we excluded
and a pole of the $(n+1)$th order at $z=e^{i\pi}=-1$,  the Cauchy residue theorem gives us
\begin{eqnarray} 
&& \displaystyle  
\ointctrclockwise\limits_{C} \frac{\,z^x\,e^{-i\pi x} \,}{\,(1+z)^n \, (\ln z - \pi i)\,}\, dz\,=\, 
2\pi i \! \res\limits_{z=-1}\! \frac{z^x\,e^{-i\pi x}}{\,(1+z)^n \, (\ln z - \pi i)\,} \,=\label{98yfdhhd}\\[2mm] 
&& \displaystyle \qquad \qquad\qquad \notag
\,=\,\frac{2\pi i}{n!}\cdot  
\left.\frac{\partial^n}{\partial z^n}\frac{z^x\,e^{-i\pi x}\, (z+1)}{\,\ln z - \pi i\,}\right|_{z=-1} 
\!\!\! = \,2\pi i \,(-1)^{n-1}\psi_n(x)  \,.
\end{eqnarray} 
Equating \eqref{89dyndhed} with \eqref{98yfdhhd}, therefore, yields
\be
\psi_n(x)\,=\,\frac{(-1)^{n+1}}{\pi} 
\!\int\limits_0^\infty \! \frac{\,\pi \cos\pi x - \sin\pi x \ln z\,}{\,(1+z)^n} 
\cdot\frac{z^x \, dz}{\,\ln^2 z +\pi^2\,}\,,
\ee
where $\,n\in\mathbb{N}\,$ and $\,-1\leqslant x\leqslant n-1\,$. 
This formula may also be written is a variety of other forms, for example,
\begin{eqnarray} 
& \displaystyle 
\psi_n(x)&=\,\frac{(-1)^{n+1}}{\pi} 
\!\int\limits_0^\infty \! \frac{\,\pi \cos\pi x + \sin\pi x \ln z\,}{\,(1+z)^n} \cdot\frac{z^{n-x-2} \, dz}{\,\ln^2 z +\pi^2\,} \\[2mm] 
&& \displaystyle 
=\,\frac{(-1)^{n+1}}{\pi} 
\!\int\limits_{-\infty}^{+\infty}\! \! \frac{\,\pi \cos\pi x - v\sin\pi x \,}{\,(1+e^v)^n} 
\cdot\frac{e^{v(x+1)} }{\,v^2 +\pi^2\,}\, dv \\[2mm] 
&& \displaystyle 
=\,\frac{(-1)^{n+1}}{\pi} 
\!\int\limits_{-\infty}^{+\infty}\! \! \frac{\,\pi \cos\pi x + v\sin\pi x \,}{\,(1+e^v)^n} 
\cdot\frac{e^{v(n-x-1)} }{\,v^2 +\pi^2\,}\, dv
\end{eqnarray} 
with the same $n$ and $x$. Note that at $x=0$ we retrieve 
Schr\"oder's formul\ae~for $G_n$ \cite{iaroslav_11}.
\qed


\small
\begin{thebibliography}{10}


\bibitem{adelberg_01}
{A.~Adelberg}, 2-Adic congruences of Nörlund numbers and
  of Bernoulli numbers of the second kind, {\it J. Number Theory},
  {\bf 73} (1998), 47--58.

\bibitem{allouche_03}
{J.-P. Allouche}, {M.~M. France} and {J.~Peyri\`ere},
  Automatic Dirichlet series, {\it J. Number Theory}, 
  {\bf 81} (2000), 359--373.

\bibitem{appel_04}
{P.~Appell}, Sur la constante $C$ d'Euler, {\it Nouv. Ann. Math (6)}, 
{\bf 1} (1925), 289--290.

\bibitem{appel_03}
{P.~Appell}, Sur les polynomes de Fontana--Bessel, {\it Enseign. Math.}, 
{\bf 25} (1926), 189--190.

\bibitem{bateman_01}
{H.~Bateman} and {A.~Erdélyi}, {\it Higher Transcendental Functions}, 
Mc Graw--Hill Book Company, 1955.

\bibitem{bessel_01}
{F.~W. Bessel}, Untersuchung der durch das Integral
  $\int\!\frac{dx}{\mathrm{l}\, x}$ ausgedrückten Function, {\it König. Arch. Naturwiss. Math.}, {\bf 1} (1811), 1--40.
  Erratum: {\bf 1} (1811), 239--240.

\bibitem{iaroslav_06}
{Ia.~V.~Blagouchine, Rediscovery of Malmsten's integrals,
  their evaluation by contour integration methods and some related results, 
  {\it Ramanujan J.}, {\bf 35} (2014), 21--110. Erratum--Addendum: {\bf 42} 
  (2017), 777--781.

\bibitem{iaroslav_07}
{Ia.~V.~Blagouchine}, A theorem for the closed--form evaluation
  of the first generalized Stieltjes constant at rational arguments and
  some related summations, {\it J. Number Theory}, {\bf 148} (2015),
  537--592. Erratum: {\bf 151} (2015), 276--277.

\bibitem{iaroslav_09}
{Ia.~V.~Blagouchine}, Expansions of generalized Euler's
  constants into the series of polynomials in $\pi^{-2}$ and into the formal
  enveloping series with rational coefficients only, {\it J. Number Theory}, 
  {\bf 158} (2016), 365--396. Corrigendum: {\bf 173} (2017), 631--632.

\bibitem{iaroslav_08}
{Ia.~V.~Blagouchine}, Two series expansions for the logarithm of
  the gamma function involving Stirling numbers and containing only
  rational coefficients for certain arguments related to $\pi^{-1}$, {\it J. Math. Anal. Appl.}, 
  {\bf 442} (2016), 404--434.

\bibitem{iaroslav_11}
{Ia.~V.~Blagouchine}, A note on some recent results for the
  Bernoulli numbers of the second kind, {\it J. Integer Seq.}, {\bf 20} 
  (2017), 1--7. Art. 17.3.8

\bibitem{bolbachan_01}
{V.~Bolbachan}, New proofs of some formulas of
  Guillera-Ser-Sondow, arXiv:0910.4048 (2009).

\bibitem{boole_01}
{G.~Boole} (J.~F.~Moulton editor), {\it Calculus of Finite Differences (4th edn.)},
Chelsea Publishing Company, New-York, USA, 1957.

\bibitem{brychkov_02}
{Yu.~A.~Brychkov}, Power expansions of powers of trigonometric
  functions and series containing Bernoulli and Euler
  polynomials, {\it Integral Transforms Spec. Funct.}, {\bf 20}
  (2009), 797--804.

\bibitem{brychkov_01}
{Yu.~A.~Brychkov}, On some properties of the generalized
  Bernoulli and Euler polynomials, {\it Integral Transforms Spec. Funct.}, 
  {\bf 23}  (2012), 723--735.

\bibitem{candelpergher_01}
{B.~Candelpergher} and {M.-A. Coppo}, A new class of identities
  involving Cauchy numbers, harmonic numbers and zeta values, 
  {\it Ramanujan J.}, {\bf 27}  (2012), 305--328.

\bibitem{candelpergher_05}
{B.~Candelpergher} and {M.-A. Coppo}, Le produit harmonique des
  suites, {\it  Enseign. Math.}, {\bf 59}  (2013), 39--72.

\bibitem{carlitz_02}
{L.~Carlitz}, Some theorems on Bernoulli numbers of higher order,
  {\it Pacific J. Math.}, {\bf 2}  (1952), 127--139.

\bibitem{carlitz_01}
{L.~Carlitz}, A note on Bernoulli and Euler polynomials of
  the second kind, {\it Scripta Math.}, {\bf 25}  (1961), 323--330.

\bibitem{chen_03}
{C.-P. Chen}, Sharp inequalities and asymptotic series of a product
  related to the Euler--Mascheroni constant, {\it J. Number Theory},
  {\bf 165}  (2016), 314--323.

\bibitem{chen_02}
{C.-P. Chen} and {R.~B. Paris}, On the asymptotic expansions of
  products related to the Wallis, Weierstrass and Wilf
  formulas, {\it Appl. Math. Comput.}, {\bf 293} (2017), 30--39.

\bibitem{coffey_08}
{M.~W. Coffey}, Addison-type series representation for the
  Stieltjes constants, {\it J. Number Theory}, 
  {\bf 130}  (2010), 2049--2064.

\bibitem{coffey_02}
{M.~W. Coffey}, Series representations for the Stieltjes
  constants, {\it Rocky Mountain J. Math.}, 
  {\bf 44}  (2014), 443--477.

\bibitem{comtet_01}
{L.~Comtet}, {\it Advanced Combinatorics. The Art of Finite and Infinite
  Expansions (revised and enlarged edn.)}, D. Reidel Publishing Company,
  Dordrecht, Holland, 1974.
  
\bibitem{connon_09}
{D.~F. Connon}, Some series and integrals involving the Riemann
  zeta function, binomial coefficients and the harmonic numbers. Volume
  II(a), arXiv:0710.4023  (2007).

\bibitem{connon_06}
{D.~F. Connon}, Some series and integrals involving the Riemann
  zeta function, binomial coefficients and the harmonic numbers. Volume
  II(b), arXiv:0710.4024  (2007).

\bibitem{connon_08}
{D.~F. Connon}, Some possible approaches to the Riemann
  hypothesis via the Li/Keiper constants, arXiv:1002.3484  (2010).

\bibitem{davis_02}
{H.~T. Davis}, The approximation of logarithmic numbers, 
{\it Amer. Math. Monthly}, {\bf 64}  (1957), part II, 11--18.

\bibitem{ahmed_02}
{L.~Fekih-Ahmed}, On the zeros of the Riemann zeta function,
  arXiv:1004.4143  (2010).

\bibitem{finch_01}
{S.~R.~Finch}, {\it Mathematical Constants}, Cambridge University Press, USA,
  2003.

\bibitem{flajolet_01}
{P.~Flajolet} and {R.~Sedgewick}, Mellin transforms and
  asymptotics: Finite differences and Rice's integrals, 
  {\it Theoret. Comput. Sci.}, {\bf 144}  (1995), 101--124.

\bibitem{franklin_01}
{F.~Franklin}, On an expression for Euler's constant, {\it John
  Hopkins Univ. Circulars}, {\bf II(25)} (1883/1888), 143.

\bibitem{goldstine_01}
{H.~H. Goldstine}, {\it A History of Numerical Analysis from the 16th through
  the 19th Century}, Springer--Verlag, New--York, Heidelberg, Berlin, 1977.

\bibitem{gould_01}
{H.~W. Gould}, Stirling number representation problems,
  {\it Proc. Amer. Math. Soc.}, 
  {\bf 11} (1960), 447--451.

\bibitem{gould_02}
{H.~W. Gould}, Note on recurrence relations for Stirling numbers,
  {\it Publ. Inst. Math. (N.S.)}, {\bf 6}  (1966),
  115--119.

\bibitem{hamming_01}
{R.~W. Hamming}, {\it Numerical Methods for Scientists and Engineers},
  McGraw--Hill Book Company, 1962.

\bibitem{hasse_01}
{H.~Hasse}, Ein Summierungsverfahren für die Riemannsche
  $\zeta$-Reihe, {\it Math. Z.}, {\bf 32} (1930), 458--464.

\bibitem{hermite_01}
{Ch.~Hermite}, Extrait de quelques lettres de M.~Ch.~Hermite à
  M.~S.~Pincherle, {\it Ann. Mat. Pura Appl. (3)}, 
 {\bf 5}  (1901), 57--72.

\bibitem{howard_03}
{F.~T. Howard}, {\it Nörlund number $B_n^{(n)}$}, in ``Applications of
  Fibonacci Numbers, {\bf 5}, 355--366'', Kluwer Academic, Dordrecht, 1993.

\bibitem{jeffreys_02}
{H.~Jeffreys} and {B.~S. Jeffreys}, {\it Methods of Mathematical
  Physics (2nd edn.)}, University Press, Cambridge, Great Britain, 1950.

\bibitem{jordan_02}
{Ch.~Jordan}, Sur des polynomes analogues aux polynomes de
  Bernoulli, et sur des formules de sommation analogues à celle de
  Maclaurin--Euler, {\it Acta Sci. Math. (Szeged)}, {\bf 4} (1928--1929),
 130--150.

\bibitem{jordan_01}
{Ch.~Jordan}, {\it The Calculus of Finite Differences (3rd edn.)}, Chelsea
  Publishing Company, New-York, USA, 1965.

\bibitem{kantorovich_01_eng}
{L.~V. Kantorovich} and {V.~I. Krylov}, {\it Approximate Methods of
  Higher Analysis (3rd edn., in Russian)}, Gosudarstvennoe izdatel'stvo
  fiziko--matematicheskoi literatury, Moscow, USSR, 1950.

\bibitem{kluyver_02}
{J.~C. Kluyver}, Euler's constant and natural numbers,
  {\it Proc. K. Ned. Akad. Wet.}, {\bf 27}  (1924), 142--144.

\bibitem{kowalenko_01}
{V.~Kowalenko}, Properties and applications of the reciprocal logarithm
  numbers, {\it Acta Appl. Math.}, {\bf 109}  (2010), 413--437.

\bibitem{krylov_01}
{V.~I. Krylov}, {\it Approximate Calculation of Integrals}, The Macmillan
  Company, New-York, USA, 1962.

\bibitem{lehmer_02}
{D.~H. Lehmer}, The sum of like powers of the zeros of the
  Riemann zeta function, {\it Math. Comp.}, {\bf 50}  (1988),
  265--273.

\bibitem{lerch_03}
{M.~Lerch}, Expressions nouvelles de la constante d'Euler,
  {\it Věstník Královské české společnosti náuk. Tř. mathematicko-přírodovědecká}, 
  {\bf 42} (1897), 1--5.

\bibitem{luke_01}
{Y.~L. Luke}, {\it The Special Functions and their Approximations}, Academic
  Press, NewYork, 1969.

\bibitem{melentiev_01_eng}
{P.~V. Melentiev}, {\it Approximate Calculations (in Russian)},
  Gosudarstvennoe izdatel'stvo fiziko--matematicheskoi literatury, Moscow,
  USSR, 1962.

\bibitem{merlini_01}
{D.~Merlini}, {R.~Sprugnoli} and M.~Cecilia
  Verri}, The Cauchy numbers, {\it Discrete Math.}, {\bf 306}  (2006),
  1906--1920.

\bibitem{milne_02}
{W.~E. Milne}, {\it Numerical Calculus: Approximations, Interpolation, Finite
  Differences, Numerical Integration and Curve Fitting}, Princeton University
  Press, Princeton, 1949.

\bibitem{milne_01}
{M.~L. Milne-Thomson}, {\it The Calculus of Finite Differences}, MacMillan and
  Co., Limited, London, 1933.

\bibitem{demorgan_01}
{A.~D. Morgan}, {\it The Differential and Integral Calculus}, Baldwin \&
  Cradock, London, 1842.

\bibitem{nemes_03}
{G.~Nemes}, Generalization of Binet's gamma function formulas,
  {\it Integral Transforms Spec. Funct.}, {\bf 24}  (2013), 597--606


\bibitem{norlund_03}
{N.~E. Nörlund}, Mémoire sur les polynomes de Bernoulli, {\it Acta
  Math.}, {\bf 43}  (1922), 121--196.

\bibitem{norlund_02}
{N.~E. Nörlund}, {\it Vorlesungen über Differenzenrechnung},
  Springer, Berlin, 1924.

\bibitem{norlund_01}
{N.~E. Nörlund}, Sur les valeurs asymptotiques des nombres et des
  polynômes de Bernoulli, {\it Rend. Circ. Mat. Palermo},
  {\bf 10} (1961), 27--44.

\bibitem{phillips_02}
{G.~M. Phillips}, {\it Interpolation and Approximation by Polynomials},
  Springer, New--York, 2003.

\bibitem{qi_01}
{F.~Qi}, An integral representation, complete monotonicity, and
  inequalities of Cauchy numbers of the second kind, {\it J. Number
  Theory}, {\bf 144}  (2014), 244--255.

\bibitem{roman_01}
{S.~Roman}, {\it The Umbral Calculus}, Academic Press, NewYork, 1984.

\bibitem{rubinstein_01}
{M.~O. Rubinstein}, Identities for the Riemann zeta function, {\it 
  Ramanujan J.}, {\bf 27}  (2012), 29--42.

\bibitem{salzer_01}
{H.~E. Salzer}, Coefficients for numerical integration with central
  differences, {\it Phil. Mag.}, {\bf 35} (1944), 262--264.

\bibitem{ser_02}
{J.~Ser}, Questions 5561--5564, {\it Intermédiaire Math. (2)},
  {\bf 4}  (1925), 126--128.

\bibitem{ser_01}
{J.~Ser}, Sur une expression de la fonction $\zeta(s)$ de
  Riemann, {\it C. R. Acad. Sci. Paris (2)}, {\bf 182} (1926), 1075--1077.

\bibitem{sitaramachandrarao_01}
{R.~Sitaramachandrarao}, Maclaurin coefficients of the Riemann
  zeta function, {\it Abstr. Papers Presented Amer. Math. Soc.}, 
  {\bf 7} (1986), 280. 

\bibitem{slavic_02}
{D.~V. Slavi\'c}, Classification with generalizations of quadrature
  formulas, {\it Univ. Beograd Publ. Elektrotehn. Fak. (Mat. Fiz.)},
  {\bf 735--762}  (1982), 102--121.

\bibitem{soldner_01}
{J.~Soldner}, {\it Théorie et Tables d'une Nouvelle Fonction Transcendante},
  Lindauer, Minuc, 1809.

\bibitem{sondow_03}
{J.~Sondow}, Analytic continuation of Riemann's zeta function and
  values at negative integers via Euler's transformation of series,
  {\it Proc. Amer. Math. Soc.}, {\bf 120}  (1994), 421--424.

\bibitem{sondow_04}
{J.~Sondow} and {S.~A. Zlobin}, Integrals over polytopes,
  multiple zeta values and polylogarithms, and Euler's constant,
  {\it Math. Notes}, {\bf 84}  (2008), 568--583.

\bibitem{hermite_02}
\'E.~Picard (editor), {\it Œuvres de Charles Hermite},
  Gauthier-Villars, Paris, 1905--1917.

\bibitem{van_veen_01}
{S.~C. Van Veen}, Asymptotic expansion of the generalized
  Bernoulli numbers $B_n^{(n-1)}$ for large values of $n$ ($n$
  integer), {\it Indag. Math.}, {\bf 13}  (1951), 335--341.

\bibitem{vacca_02}
{G.~Vacca}, Sulla costante di Eulero, $C=0,577\ldots$, 
{\it Atti Accad. Naz. Lincei, Classe di scienze fisiche, matematiche
  e naturali (6)}, {\bf 1} (1925), 206--210.

\bibitem{vorobiev_01}
{N.~N. Vorobiev}, {\it Theory of Series (4th edn., in Russian)},
 Nauka, Moscow, USSR, 1979.

\bibitem{weisstein_04}
{E.~W. Weisstein}, Euler-mascheroni constant, From MathWorld--A Wolfram
  Web Resource, published electronically at\\
  {\tt http://mathworld.wolfram.com/Euler-MascheroniConstant.html}

\bibitem{whittaker_02}
{E.~Whittaker} and {G.~Robinson}, {\it The Calculus of Observations. A
  Treatise on Numerical Mathematics}, Blackie and Son Limited, London, 1924.

\bibitem{xu_01}
{C.~Xu}, {Y.~Yan} and {Z.~Shi}, Euler sums and integrals
  of polylogarithm functions, {\it J. Number Theory}, 
  {\bf 165} (2016), 84--108.

\bibitem{young_02}
{P.~T. Young}, Congruences for Bernoulli, Euler and
  Stirling numbers, {\it J. Number Theory},
  {\bf 78} (1999), 204--227.

\bibitem{young_03}
{P.~T. Young}, Rational series for multiple zeta and log gamma
  functions, {\it J. Number Theory}, 
  {\bf 133}  (2013), 3995--4009.

\bibitem{zhao_01}
{F.-Z. Zhao}, Sums of products of Cauchy numbers, {\it Discrete
  Math.}, {\bf 309}  (2009), 3830--3842.

\end{thebibliography}
\end{document}